\documentclass[11pt,reqno]{amsart}
\usepackage[T1]{fontenc}
\usepackage[utf8]{inputenc}
\usepackage{newtxtext}
\usepackage{amsmath,amssymb,mathtools,newtxmath}
\usepackage[margin=1in]{geometry}
\usepackage{microtype,enumitem,algorithm,algpseudocode,needspace,booktabs,array,tabularx}
\usepackage[section]{placeins}
\usepackage{graphicx}
\usepackage{xcolor}
\usepackage{hyperref}
\IfFileExists{bookmark.sty}{\usepackage{bookmark}}{}
\hypersetup{pdftitle={Exact construction of quasi-uniform designs over polytopal domains},pdfauthor={S. Yang and T. Goda}}
\setlist[enumerate]{leftmargin=2em,itemsep=.2em,topsep=.3em}
\setlist[itemize]{leftmargin=1.6em,itemsep=.1em,topsep=.2em}
\allowdisplaybreaks[1]
\newtheorem{theorem}{Theorem}[section]
\newtheorem{lemma}[theorem]{Lemma}
\newtheorem{proposition}[theorem]{Proposition}
\newtheorem{corollary}[theorem]{Corollary}
\theoremstyle{definition}

\newtheorem{problem}{Problem}
\theoremstyle{remark}

\newtheorem{example}[theorem]{Example}
\newcommand{\R}{\mathbb R}
\newcommand{\Q}{\mathbb Q}
\newcommand{\Z}{\mathbb Z}
\newcommand{\K}{\mathcal K}
\newcommand{\Far}{\mathcal F}
\newcommand{\dist}{\operatorname{dist}}
\newcommand{\vol}{\operatorname{vol}}

\newcommand{\Vertx}{\operatorname{Vert}}
\newcommand{\lexmin}{\operatorname{lexmin}}
\newcommand{\argmax}{\operatorname*{arg\,max}}
\newcommand{\bl}{\operatorname{bl}}
\newcommand{\height}{\operatorname{ht}}

\newcommand{\LEFPS}{\mathrm{LEFPS}}
\newcommand{\AAS}{\mathrm{AAS}}

\newcommand{\conv}{\operatorname{conv}}

\newcommand{\doi}[1]{\href{https://doi.org/#1}{\nolinkurl{#1}}}

\mathtoolsset{showonlyrefs}
\allowdisplaybreaks

\numberwithin{equation}{section}

\title[Quasi-uniform designs over polytopal domains]{Exact construction of quasi-uniform designs over polytopal domains}
\author{Songlin Yang}
\address{Graduate School of Engineering, The University of Tokyo,
7-3-1 Hongo, Bunkyo-ku, Tokyo 113-8656, Japan}
\email{ysl@g.ecc.u-tokyo.ac.jp}
\author{Takashi Goda}
\address{Graduate School of Information Science and Technology, The University of Tokyo,
7-3-1 Hongo, Bunkyo-ku, Tokyo 113-8656, Japan}
\email{takashigoda@g.ecc.u-tokyo.ac.jp}
\thanks{The second author acknowledges the support provided through ChatGPT for academic researchers (OpenAI) in revising this manuscript, particularly in improving clarity and readability and reducing repetition.}

\date{\today}

\begin{document}

\begin{abstract}
We develop two sampling algorithms for constructing extensible designs with mesh ratio at most two on finite unions of full-dimensional rational convex polytopes. Starting from a rational initial point, both algorithms produce rational coordinates and perform the geometric comparisons governing point selection exactly. This ensures that distance ties and geometric degeneracies are handled without rounding ambiguity, preserving the prescribed selection rules and mesh-ratio guarantee at every stage. The first algorithm, Local Exact Farthest-Point Sampling (LEFPS), implements a prescribed farthest-point rule by local updates of the restricted Voronoi diagram. The second algorithm, Arithmetic-Aware Admissible Sampling (AAS), minimizes the representation length of reduced rational coordinates over admissible fragment vertices and edge midpoints. Since the cost of exact arithmetic depends on coordinate lengths, we analyze both operation counts and bit complexity. For fixed dimension and input, each sampling rule admits an $O(N\log N)$ unit-cost implementation and a separate implementation with $\widetilde O(L_N)$ bit complexity, where $L_N$ is the total length of the output coordinates. On rational interval covers, every admissible sequence from the vertex-midpoint candidate set has logarithmic coordinate height. We also identify the ordered sequences explicitly for LEFPS on interval unions and four-dimensional coordinate-order domains, and for AAS on a rectangle and a nonconvex cross. These descriptions yield direct generation with $O(N)$ unit-cost work and optimal $\Theta(N\log N)$ bit work, including exact quality measures at every prefix. Numerical experiments compare local Voronoi maintenance, coordinate growth, and direct generation.
\end{abstract}
\maketitle

\section{Introduction}\label{sec:intro}
Space-filling designs are used in computer experiments and function approximation to distribute sampling points over a prescribed domain. Two geometric properties are particularly desirable: the points should cover the domain well and remain sufficiently separated from one another. These properties are closely related to approximation accuracy and numerical stability. For example, in kernel interpolation, the covering radius enters error estimates, while the separation radius enters stability estimates under suitable assumptions on the kernel and the domain; see, e.g., \cite{PZ,WenzelEtAl}. When the evaluation budget is not fixed in advance, the design should also be extensible: new points should be added without discarding existing points or their function values, while preserving geometric quality at every stage.

Let $\Omega\subset\R^d$ be a compact set of positive volume, equipped with the Euclidean norm $\|\cdot\|$. For an infinite sequence of distinct points $(x_n)_{n\ge1}\subset\Omega$, let $P_N=\{x_1,\ldots,x_N\}$ denote the set of the first $N$ points. For a nonempty finite set $P\subset\Omega$, we write $d_P(y)=\min_{x\in P}\|y-x\|$. The \emph{covering radius} and \emph{separation radius} of $P_N$ are defined by
\[
 h(P_N)=\max_{y\in\Omega}d_{P_N}(y),\qquad
 q(P_N)=\frac12\min_{1\le i<j\le N}\|x_i-x_j\|,
\]
respectively, where the latter is defined for $N\ge2$. We suppress the dependence of $h(P_N)$ on the fixed domain $\Omega$, and write $\delta(P_N)=2q(P_N)$ for the minimum pairwise distance. Their ratio $\rho(P_N)=h(P_N)/q(P_N)$ is the \emph{mesh ratio}. The designs are \emph{quasi-uniform} if $\sup_{N\ge2}\rho(P_N)<\infty$. For a fixed domain, this condition implies that both radii have the order $N^{-1/d}$; see \cite[Corollary~2.1]{PZ}. Thus, quasi-uniformity provides a common scale for coverage and separation, uniformly in $N$.

\subsection{Farthest-point sampling and existing constructions}

\emph{Farthest-point sampling}, also called greedy packing, selects a point farthest from the current design; related constructions appear in clustering and progressive image sampling \cite{Gonzalez,Eldar}. From $x_1\in\Omega$, we use the deterministic rule
\begin{equation}\label{eq:canonical}
 x_{n+1}=\lexmin\argmax_{y\in\Omega}d_{P_n}(y),\qquad n\ge1.
\end{equation}
Here, lexicographic order in the given coordinates resolves ties. The insertion distances are non-increasing, so $2q(P_N)=h(P_{N-1})\ge h(P_N)$ and $\rho(P_N)\le2$. This constant is optimal: every infinite nested sequence on a compact set of positive volume has $\limsup_N\rho(P_N)\ge2$, as shown in \cite[Theorems~1.1 and~3.2]{PZ}.

The difficulty of farthest-point sampling is maximizing $d_{P_n}$ over the whole domain at each insertion. On polytopal domains, restricted Voronoi diagrams provide a finite set of candidate points. Eldar et al.\ \cite[Sections~III--IV]{Eldar} obtain an $O(N\log N)$ implementation on rectangles under their initialization and geometric assumptions. Xu and Goda~\cite{XuGoda} develop a Voronoi-guided construction on arbitrary triangles and discuss the same operation count. These results motivate local Voronoi updates on more general polytopal domains.

An alternative is to apply farthest-point sampling to a finite random candidate set. Under suitable domain regularity, Goda and Xu~\cite{GodaXuRandom} prove that $N\log N$ is the sharp order of the number of independent candidates needed to select $N$ points with bounded mesh ratio with high probability, and give an extensible randomized construction. Here we use an explicit polyhedral description of the domain to obtain deterministic constructions with mesh ratio at most two.

\subsection{Exact construction and arithmetic cost}

A geometric selection rule does not by itself specify how its decisions can be carried out reliably on a computer. In \eqref{eq:canonical}, one must compare candidate distances, detect ties, and update the Voronoi structure, including in degenerate configurations. Errors in these decisions can change the selected point or invalidate the maintained geometry. The importance of exact geometric predicates for avoiding such inconsistencies is well established in computational geometry; see, for example, \cite{Shewchuk}. The aim of this paper is to implement the specified selection rules exactly and retain their mesh-ratio guarantee at every stage.

We consider domains
\[ \Omega=\bigcup_{\sigma\in\K}\sigma, \] 
where $\K$ is a finite family of full-dimensional compact rational convex polytopes, each given by rational linear inequalities. The cells may overlap, and the union may be nonconvex or disconnected. Distances are measured in the ambient Euclidean space. Starting from a rational point, this setting permits exact construction using rational arithmetic. Indeed, the bisector of two rational sites $p$ and $q$ has equation
\[ 2(q-p)\cdot y = \|q\|^2-\|p\|^2.\]
Restricted Voronoi vertices therefore have rational coordinates, as do edge midpoints, and distances can be compared through their rational squares. In this paper, \emph{exact construction} means that the output coordinates are represented as reduced rational numbers and the geometric comparisons governing point selection are performed exactly.

Exactness, however, does not make arithmetic operations uniformly inexpensive. Their cost depends on the lengths of the numerators and denominators, which may grow as points are inserted. An operation count that assigns unit cost to every rational operation can therefore conceal a substantial part of the computational work. This also suggests considering coordinate length when choosing a new point: the mesh-ratio bound may allow an insertion with shorter coordinates than a farthest point.

We accordingly analyze both unit-cost operation counts and bit complexity. Let $\lambda(x)$ denote the representation length of the reduced rational coordinates of $x$, as defined precisely in \eqref{eq:output-size}, and put
\[
 L_N=\sum_{n=1}^N\lambda(x_n).
\]
The total output length $L_N$ is a lower bound, up to a constant factor, on the bit complexity of any algorithm that explicitly writes the reduced coordinates. A bound close to this output cost measures the efficiency of executing a selection rule; a bound in terms of $N$ alone additionally requires control of coordinate growth. Throughout the analysis, the dimension, the given polytopal cover, and the rational initial point are fixed as $N$ grows.

\subsection{Main contributions}

We develop two exact sampling algorithms and analyze their geometric and arithmetic costs. The main contributions are as follows.

\begin{enumerate}[label=(\roman*)]
\item \textbf{Local Exact Farthest-Point Sampling (LEFPS).}
We implement \eqref{eq:canonical} by updating only affected restricted Voronoi fragments, the intersections of Voronoi cells with input polytopes. The mesh-ratio bound limits each update to a bounded number of neighboring sites. The construction handles degeneracies and overlapping input cells, and its ordered output is independent of the chosen cover.

\item \textbf{Arithmetic-Aware Admissible Sampling (AAS).}
For a singleton or a design with $h(P)\le2q(P)$, an insertion $x\in\Omega\setminus P$ preserves this bound exactly when
\begin{equation}\label{eq:intro-admissible}
 h(P\cup\{x\})\le d_P(x).
\end{equation}
We call such an insertion \emph{admissible}. AAS minimizes $\lambda(x)$, with lexicographic ties, over admissible fragment vertices and edge midpoints. We give an exact covering test and show that the minimum over the entire candidate set can be found by testing only boundedly many local candidates. The candidate set, and hence the output, may depend on the cover. The minimization concerns the current candidates; it does not imply a smaller total output length than LEFPS.

\item \textbf{Execution cost and coordinate growth.}
Each sampling rule admits an $O(N\log N)$ unit-cost implementation and a separate implementation with $\widetilde O(L_N)$ bit complexity, including reduced point output. Here $\widetilde O$ suppresses logarithmic factors in $N$ and $L_N$. The bit bound accounts for local arithmetic, repeated use of coordinates, and global ordering operations. On every fixed rational interval cover, every admissible sequence from the vertex--midpoint candidate set satisfies $\lambda(x_n)=O(\log(n+2))$ and $L_N=\Theta(N\log N)$. Both algorithms consequently admit $O(N\log^2N)$ bit implementations on these domains. Coordinate growth on general higher-dimensional domains remains open.

\item \textbf{Direct Stream Execution (DSE).}
For particular inputs, we identify the ordered point sequences explicitly. These descriptions reproduce the same points in the same order with $O(N)$ unit-cost work and optimal $\Theta(N\log N)$ bit work, including exact squared covering and separation radii. The bounds hold at every stopping point, including setup and preparation costs. We treat LEFPS on rational interval unions from any rational singleton, and on four-dimensional coordinate-order domains \eqref{eq:orderdomain} from the origin, including rational similarity images with the corresponding initial points. For AAS, we treat the rectangle and nonconvex cross in \eqref{eq:rectangle-input} and \eqref{eq:cross} with their specified covers and initial points. The planar and four-dimensional constructions build on the square and four-cube refinements of Pronzato and Zhigljavsky~\cite[Theorems~3.3 and~3.5]{PZ}; we establish agreement with our selection rules and the generation cost at every prefix.
\end{enumerate}

The rest of the paper is organized as follows. Section~\ref{sec:background} introduces notation and computational conventions. Sections~\ref{sec:geometry} and~\ref{sec:admissible} present LEFPS and AAS, respectively. Section~\ref{sec:arithmetic} analyzes bit complexity and coordinate growth, and Section~\ref{sec:direct} develops direct generation. Section~\ref{sec:experiments} conducts numerical experiments, comparing the costs of local Voronoi maintenance, the two selection rules, and direct generation. Section~\ref{sec:discussion} discusses the remaining questions. Appendix~\ref{app:direct-proofs} contains detailed proofs for the planar AAS constructions, including coordinate-length comparisons, verification of the initial sequences, and the effect of rational scaling on the refinement order.

\section{Preliminaries}\label{sec:background}

\subsection{Domains and geometric scales}\label{subsec:domains-scales}

We consider the rational polytopal domains introduced in Section~\ref{sec:intro}, represented as
\begin{equation}\label{eq:cover}
 \Omega=\bigcup_{\sigma\in\K}\sigma\subset\R^d,
\end{equation}
where $d\ge1$ and $\K$ is a nonempty finite family of full-dimensional compact convex polytopes given by rational affine inequalities. These \emph{input cells} may overlap without meeting face to face. All distances, including those between components, are ambient Euclidean distances.

The fixed input $I$ consists of these inequalities and the rational seed $x_1\in\Omega \cap \Q^d$. Complexity constants may depend on $I$ and $d$, but not on $N$. One-time exact preprocessing removes redundant inequalities, computes cell vertices, and finds a rational enclosing box.

We use $h(P)$, $q(P)$, and $\delta(P)=2q(P)$ as in Section~\ref{sec:intro}. A design with at least two points is called \emph{certified} if $h(P)\le\delta(P)$, equivalently $\rho(P)\le2$. For induction, set $q(\{x_1\})=\delta(\{x_1\})=+\infty$ without assigning a mesh ratio to a singleton. Write $h_n=h(P_n)$ and $r_{n+1}=d_{P_n}(x_{n+1})$ for the \emph{insertion radius}. Adding a point cannot increase $h_n$, and
\[
 \delta(P_{n+1})=\min\{\delta(P_n),r_{n+1}\}.
\]
For farthest-point sampling, $r_{n+1}=h_n$ is nonincreasing, so
\begin{equation}\label{eq:greedy-separation}
 2q(P_n)=\delta(P_n)=h_{n-1}\ge h_n,\qquad n\ge2.
\end{equation}

Throughout this paper, $B(x,r)$ denotes the closed Euclidean ball with center $x$ and radius $r$. Let $v_d=\vol(B(0,1))$. For every $n$-point design with $n\ge2$,
\begin{equation}\label{eq:volume-scales}
 h(P_n)\ge\left(\frac{\vol(\Omega)}{nv_d}\right)^{1/d},
 \qquad
 \delta(P_n)\le C_\Omega n^{-1/d}.
\end{equation}
The first bound follows by covering $\Omega$ with $n$ balls of radius $h(P_n)$. For the second, the open balls of radius $q(P_n)$ centered at the design points are disjoint and lie in the Minkowski sum of the sets $\Omega$ and $B(0,q(P_n))$. Since $q(P_n)\le \operatorname{diam}(\Omega)/2$ for $n\ge2$, this enlarged domain has uniformly bounded volume, giving the second bound by volume comparison. For certified sequences, the two bounds and $h(P_n)\le2q(P_n)$ yield $h_n\asymp q(P_n)\asymp n^{-1/d}$, with constants independent of $n$.

\subsection{Rational coordinates and computational models}
\label{subsec:input-output}\label{subsec:cost-models}

For an integer $a\ne0$, let $\bl(a)=\lfloor\log_2|a|\rfloor+1$, and set $\bl(0)=0$. For a reduced rational $u=a/b$ with $b>0$, define
\[
 \height(u)=\max\{\bl(a),\bl(b)\},\qquad
 \height(x)=\max_{1\le j\le d}\height(x_j)\quad(x\in\Q^d).
\]
If $x_{ij}=a_{ij}/b_{ij}$ is reduced with $b_{ij}>0$, set
\begin{equation}\label{eq:output-size}
 \lambda_i=\lambda(x_i)
 =1+\sum_{j=1}^d\bigl(\bl(a_{ij})+\bl(b_{ij})\bigr),
 \qquad L_N=\sum_{i=1}^N\lambda_i.
\end{equation}
Thus $\lambda(x)\asymp_d1+\height(x)$. Including signs and delimiters between integers changes the encoding length by at most a constant factor, so $L_N$ measures the size of the explicit reduced-coordinate output. AAS uses the precise value of $\lambda$ in \eqref{eq:output-size} to compare candidates, independently of the output format. Superscripts on $L_N$ distinguish the generated sequences when needed.

In the \emph{unit-cost model}, each exact rational arithmetic operation or comparison has unit cost, independently of operand lengths. Accessing or updating a record containing a bounded number of rational numbers, integers, or pointers also has unit cost. This count is also called \emph{structural work}. Exact integer division is available for spatial indexing, and each reduced rational record stores the binary lengths of its numerator and denominator. In the \emph{bit-complexity model}, costs depend on data lengths and include initialization, arithmetic, reduction, encoding, data access, and output. The arithmetic and data-access bounds are specified in Section~\ref{sec:arithmetic}. For LEFPS and AAS, the two complexity bounds may require different implementations of the same selection rule; these produce identical ordered points.

An algorithm is \emph{autonomous} if it uses only the fixed input and generated history, outputs each point after insertion, and can stop after any $P_N$ without inserting $x_{N+1}$ or using a prescribed final $N$ to define the sequence. Points retain their insertion indices as identifiers; lexicographic ties refer to the original coordinates. Unless stated otherwise, the LEFPS and AAS bounds include explicit reduced-coordinate output and one final export of the maintained state sufficient to resume the construction. They do not include repeated exports of the geometry or output of $h_n^2$ and $\delta(P_n)^2$ at every insertion. The output conventions for direct generation are specified in Section~\ref{sec:direct}.

\subsection{Restricted Voronoi diagrams and locality}
\label{subsec:voronoi-witnesses}\label{subsec:local-geometry}

Let $P\subset\Omega$ be nonempty and finite. Its points are also called \emph{sites}. For distinct $p,q\in P$, define
\begin{equation}\label{eq:halfspace}
 \begin{aligned}
  H_{p\preceq q}
   &=\{y\in\R^d:2(q-p)\cdot y\le\|q\|^2-\|p\|^2\},\\
  V_p(P)&=\bigcap_{q\in P\setminus\{p\}}H_{p\preceq q},
  \qquad F_{\sigma,p}(P)=\sigma\cap V_p(P).
 \end{aligned}
\end{equation}
For a singleton, $V_p(P)=\R^d$. A nonempty $F_{\sigma,p}(P)$ is called a \emph{Voronoi fragment}, with \emph{owner} $p$, and is \emph{essential} if it has dimension $d$. We omit $P$ when it is clear from the context. Fragments are stored separately for each pair $(\sigma,p)$, with their defining constraints and face incidences; a shared vertex may therefore have several stored occurrences. An \emph{input facet} lies in a facet hyperplane of the fragment's own input cell $\sigma$, even if this hyperplane is internal to the cover. All other facets are \emph{non-input facets}.

The following theorem extends the finite-candidate characterization for triangles in \cite[Lemma~2.8]{XuGoda} to finite polytopal covers in arbitrary dimension.

\begin{theorem}[Farthest points among fragment vertices]\label{thm:witness}
Every point of each input cell $\sigma$ lies in an essential fragment $F_{\sigma,p}$. Moreover,
\begin{equation}\label{eq:complete-far}
 \Far(P):=\argmax_{y\in\Omega}d_P(y)
 =\bigcup_{\substack{\sigma\in\K,\ p\in P\\\dim F_{\sigma,p}=d}}
   \{v\in\Vertx(F_{\sigma,p}):\|v-p\|=h(P)\}.
\end{equation}
In particular, $\Far(P)$ is finite. For $t\ge0$, the inequality $h(P)^2\le t$ holds if and only if $\|v-p\|^2\le t$ at every vertex of every essential fragment with owner $p$.
\end{theorem}
\begin{proof}
For $y\in\sigma$, choose points in $\operatorname{int}(\sigma)$ converging to $y$ and avoiding all site bisectors. A subsequence has the same unique nearest site $p$. Each of its points has a neighborhood contained in $F_{\sigma,p}$, so this fragment is full-dimensional; closedness places $y$ in it.

On $F_{\sigma,p}$, we have $d_P(z)^2=\|z-p\|^2$. If a farthest point were not a vertex of an essential fragment containing it, it would be a nontrivial convex combination of two distinct fragment points. Strict convexity of the squared norm would then contradict maximality. Conversely, each vertex on the right-hand side of \eqref{eq:complete-far} is a farthest point. Finiteness and the threshold statement follow.
\end{proof}

The theorem requires no general-position assumption, i.e., the theorem remains valid when more than $d$ constraint hyperplanes meet at a vertex or when Voronoi boundaries coincide with input-cell facets. With rational sites, all fragments and their vertices are rational, so covering radii can be compared exactly through squared owner--vertex distances.

For an insertion $x\in\Omega\setminus P$, put $P^+=P\cup\{x\}$ and $H=h(P)$. Then
\begin{equation}\label{eq:locality}
  V_x(P^+)\cap\Omega \subset B(x,H),\qquad \{p\in P:V_p(P)\cap V_x(P^+)\cap\Omega\ne\emptyset \}
    \subset P\cap B(x,2H).
\end{equation}
Indeed, on the new cell, $\|y-x\|\le d_P(y)\le H$. If $p$ is an old nearest site at such a point, the triangle inequality gives $\|p-x\|\le2H$. The second set contains every old owner whose fragments can change, including equality contacts. An upper bound for $H$ may replace $H$ when searching for these owners.

\begin{lemma}[Bounded local geometric complexity]\label{lem:compact}
Suppose $h(P)\le s$ and $\delta(P)\ge cs$ for $s>0$ and fixed $c>0$. Every essential fragment has a bounded number of facets, vertices, and face incidences. For fixed $A>0$, a ball of radius $As$ contains a bounded number of sites and fragment-vertex occurrences, and intersects fragments belonging to a bounded number of owners. An insertion with radius at least $cs$ replaces only a bounded number of geometric records. The bounds depend only on $\K,d,c,A$, not on $P$ or $s$. The total number of essential fragments is between $|P|$ and $|\K||P|$.
\end{lemma}
\begin{proof}
An owner-$p$ fragment lies in $B(p,s)$. If a site $q$ supports one of its Voronoi facets, a point on that facet is within $s$ of both sites, so $\|p-q\|\le2s$. The disjoint open balls of radius $cs/2$ centered at sites bound the number of such neighbors. Together with the fixed input facets, this bounds all face and incidence counts in fixed dimension.

If a fragment meets $B(a,As)$, its owner lies in $B(a,(A+1)s)$. The same packing argument bounds the number of these owners and hence their fragments and vertex occurrences. By \eqref{eq:locality}, only owners in $B(x,2s)$ can be affected by an insertion. Its radius assumption preserves separation at least $cs$, while the covering radius cannot increase, so the bounds apply both before and after the update. Finally, each site $p$ is uniquely nearest in a sufficiently small neighborhood of itself. Intersecting this neighborhood with any full-dimensional input cell containing $p$ shows that $p$ owns at least one essential fragment. Since it owns at most one fragment per input cell, the total lies between $|P|$ and $|\K||P|$.
\end{proof}

Whenever an essential fragment is updated or created, redundant site inequalities are discarded. The fragment is stored using the fixed input inequalities and only the nonredundant site inequalities supporting its current facets, together with its vertices and face incidences. Local polyhedral computations enumerate independent sets of $d$ constraint equations, solve the resulting rational systems, and test feasibility and incidences. This procedure also applies when the resulting polytope is lower-dimensional: at every vertex, the active constraint normals span $\R^d$, since otherwise a short feasible segment through that point would contradict extremality. Reused vertex coordinates retain a valid basis of current supporting equations. See \cite{Aurenhammer,Shewchuk} for background on Voronoi structures and geometric predicates.

Lemma~\ref{lem:compact} bounds the number of operations in each local geometric computation; finding the relevant records and maintaining their global order require the data structures in Sections~\ref{sec:geometry} and~\ref{sec:admissible}. The constants need not depend polynomially on $d$; for norm maximization in variable dimension, see \cite{BodlaenderEtAl,KnauerEtAl}.

\section{Local Exact Farthest-Point Sampling}\label{sec:geometry}

Our first algorithm, LEFPS, selects the lexicographically first farthest point by maintaining the essential Voronoi fragments and their maximizing vertices. We first show how to reconstruct the fragments affected by an insertion, then use this update to maintain the farthest point. The reconstruction applies to any insertion and is also used by our second algorithm, AAS.

\subsection{Local updates of the restricted Voronoi diagram}
\label{subsec:local-reconstruction}

For an insertion $x\in\Omega\setminus P$, put $P^+=P\cup\{x\}$ and
\[
 \ell_{x,p}(y)=\|y-p\|^2-\|y-x\|^2
 =2(x-p)\cdot y+\|p\|^2-\|x\|^2.
\]
An old essential fragment $F=F_{\sigma,p}(P)$ retains its part in $H_{p\preceq x}=\{\ell_{x,p}\le0\}$ and transfers its part in $H_{x\preceq p}=\{\ell_{x,p}\ge0\}$ to $x$. Compute $m_F=\min_F\ell_{x,p}$ and $M_F=\max_F\ell_{x,p}$ from its vertices. Since $F$ is full-dimensional and $x\ne p$, $m_F<M_F$, and
\begin{equation}\label{eq:donor-signs}
\begin{array}{ll}
M_F\le0:&F\text{ is retained},\\
m_F\ge0:&\text{the retained part is empty or lower-dimensional},\\
m_F<0<M_F:&\text{both parts are full-dimensional}.
\end{array}
\end{equation}
An old essential fragment is a \emph{donor} if it transfers positive volume to $x$, and a \emph{proper donor} if it also retains positive volume. For each input cell, define
\[
\begin{aligned}
D_\sigma(x)&=\{p:\dim F_{\sigma,p}=d,\ M_{F_{\sigma,p}}>0\},\\
J_\sigma(x)&=\{p:\dim F_{\sigma,p}=d,\ m_{F_{\sigma,p}}<0<M_{F_{\sigma,p}}\},\\
\Pi_{x,p}&=\{\ell_{x,p}=0\},\qquad
I_{\sigma,p}=F_{\sigma,p}(P)\cap\Pi_{x,p}.
\end{aligned}
\]
Here $D_\sigma(x)$ and $J_\sigma(x)$ contain the owners of donors and proper donors, respectively. The interfaces of the proper donors determine the new fragment.

\begin{theorem}[Reconstruction from the affected fragments]
\label{thm:reconstruct}
For every old essential fragment, the updated fragment is $F_{\sigma,p}(P)\cap H_{p\preceq x}$. The new fragment $C=F_{\sigma,x}(P^+)$ is essential if and only if $D_\sigma(x)\ne \emptyset$. Whenever it is essential,
\begin{align}
 C&=\sigma\cap\bigcap_{p\in D_\sigma(x)}H_{x\preceq p}
   =\sigma\cap\bigcap_{p\in J_\sigma(x)}H_{x\preceq p},
      \label{eq:donor-reconstruct}\\
 \Vertx(C)&=(\Vertx(\sigma)\cap C)
       \cup\bigcup_{p\in J_\sigma(x)}\Vertx(I_{\sigma,p}).
      \label{eq:vertex-reuse}
\end{align}
The non-input facets of $C$ are precisely the interfaces $I_{\sigma,p}$ with $p\in J_\sigma(x)$. Consequently, none of the proper-donor inequalities in the second description of \eqref{eq:donor-reconstruct} can be omitted while keeping the input inequalities. The owner of every donor fragment lies in $B(x,2h(P))$.
\end{theorem}

\begin{proof}
The nearest-site inequalities and Theorem~\ref{thm:witness} give
\[
 F_{\sigma,p}(P^+)=F_{\sigma,p}(P)\cap H_{p\preceq x},\qquad
 C=\bigcup_{\dim F_{\sigma,p}=d}(F_{\sigma,p}\cap H_{x\preceq p}).
\]
Nonessential old fragments only shrink. The finite union has positive volume exactly when a donor exists, and \eqref{eq:locality} places the owner of every donor fragment in $B(x,2h(P))$.

Suppose $C$ is essential. A non-input facet is supported by some $\Pi_{x,p}$. Choose a relative-interior point in $\operatorname{int}(\sigma)$ avoiding the other bisectors with $x$. For fixed $x$, distinct sites give distinct bisector hyperplanes: reflection of $x$ across $\Pi_{x,p}$ uniquely determines $p$. Their intersections therefore have dimension at most $d-2$. At the chosen point,
\[
 \|y-p\|=\|y-x\|<\|y-q\| \qquad (q\ne p).
\]
Thus $p$ is the unique old nearest site on a neighborhood crossing the facet, so $p\in J_\sigma(x)$. Conversely, for a proper donor the bisector crosses the fragment interior, and the nearest-site inequalities give $I_{\sigma,p}=C\cap\Pi_{x,p}$, a non-input facet. Together with the input inequalities, these facet inequalities describe $C$, and each proper-donor inequality is necessary. Every other donor inequality is valid on $C$, so adding it leaves $C$ unchanged. This proves \eqref{eq:donor-reconstruct}. A bisector coinciding with an input facet requires no additional inequality, since that boundary is already imposed by $\sigma$.

A vertex of $C$ on a non-input facet is an interface vertex. At any other vertex, the active input normals span $\R^d$, so it is an input-cell vertex. Conversely, interface vertices and input-cell vertices retained in $C$ are extreme in $C$. This proves \eqref{eq:vertex-reuse}.
\end{proof}

Strict signs exclude cases such as $m_F=0<M_F$, where the retained part is lower-dimensional. If $C$ is essential but has no proper donors, then $C=\sigma$. The retained cuts already supply the interface vertices, whose coordinates are reused in \eqref{eq:vertex-reuse}. Each input cell is updated separately, so overlaps between cells require no special treatment.

\subsection{Maintaining the farthest point}\label{subsec:lefps-maxima}

For each essential fragment $F=F_{\sigma,p}$, store the squared distances $\|v-p\|^2$ at its vertices and the pair
\begin{equation}\label{eq:record}
 \mu_F^2=\max_{v\in\Vertx(F)}\|v-p\|^2,
 \qquad
 c_F=\lexmin\argmax_{v\in\Vertx(F)}\|v-p\|^2.
\end{equation}
We call $(\mu_F^2,c_F)$ its \emph{maximum record}. Maintain these records in decreasing order of $\mu_F^2$, breaking ties first by the lexicographic order of $c_F$, and then by fixed cell and owner identifiers. The identifiers distinguish multiple occurrences of the same geometric point; they never override the prescribed coordinate tie rule.

By Theorem~\ref{thm:witness}, the first record gives both the squared covering radius and the next point in \eqref{eq:canonical}. The following lemma allows stored distances and records to be reused.

\begin{lemma}[Reuse of vertex distances and maximum records]
\label{lem:persistence}
The first maximum record determines the farthest point selected by \eqref{eq:canonical}, independently of the polytopal representation of $\Omega$. After cutting an old fragment by $H_{p\preceq x}$, the maximum record of an essential retained fragment can be obtained from the surviving old vertices and the newly created cut vertices. Squared distances to the old owner do not change at surviving vertices, and an interface vertex has the same squared distance to the old and new owners. In particular, if the old maximizing vertex $c_F$ survives, the maximum record does not require a numerical or ordering update.
\end{lemma}
\begin{proof}
Theorem~\ref{thm:witness} places every farthest point at a fragment vertex, so the first record selects the lexicographically smallest one, independently of the cover. A halfspace cut retains old vertices and creates only edge--hyperplane intersections. Surviving owner distances are unchanged, and bisector vertices are equidistant from the two owners. Restriction cannot increase the maximum. If $c_F$ survives, the old maximum is still attained, and every maximizing point of the retained fragment was already a maximizing point of $F$. Thus neither the value nor its lexicographically first maximizer changes.
\end{proof}

Keep the vertex distances ordered within each fragment, evaluate each new interface distance once, and change the global index only when a maximum record changes, is created, or is removed. For each reused vertex, update its incident faces and retain $d$ independent active constraint equations, selected by exact rank tests. These equations are used in the bit-complexity analysis.

\subsection{Algorithm and computational complexity}
\label{subsec:lefps-algorithm}

Maintain a grid of side $s$ with $s/2<h(P)\le s$ in a fixed enclosing box. Inspecting a bounded number of grid cubes finds all owners in $B(x,2h(P))$, possibly together with additional nearby sites. For certified $P$, Lemma~\ref{lem:compact} bounds the report size; singleton work depends only on the input.

Rebuild the grid only when $h(P)\le s/2$, choosing the new dyadic scale so that $s/2<h(P)\le s$ again and skipping unused intermediate scales. If $n_j,s_j$ are the populations and scales at rebuilds, then through $P_N$,
\begin{equation}\label{eq:grid-sum}
 \sum_j n_j=O(N),\qquad \sum_j s_j^{-d}=O(N).
\end{equation}
Indeed, certification gives $\delta(P)>s_j/2$ at each rebuild, so packing yields $n_j\le Cs_j^{-d}$. The scales decrease geometrically, and the covering lower bound in \eqref{eq:volume-scales} gives $s_j^{-d}=O(N)$ at the final scale. Initializing grids and reassigning sites therefore cost $O(N)$ altogether. Squared dyadic thresholds avoid square roots. Algorithm~\ref{alg:lfpm} combines this grid with local reconstruction and the maximum index.

\begin{algorithm}[htbp]
\caption{Local Exact Farthest-Point Sampling (LEFPS)}\label{alg:lfpm}
\begin{algorithmic}[1]
\Require Rational polytopal cover $\K$ and initial point
         $x_1\in\Omega\cap\Q^d$.
\State Set $P\gets\{x_1\}$; initialize the essential fragments,
       vertex distances, maximum records, and spatial index.
\State Output $x_1$ in reduced rational coordinates.
\While{another point is requested}
 \State Read $x$ and $H^2=h(P)^2$ from the first maximum record.
 \State Report nearby owners, including all sites in $B(x,2H)$.
 \State Classify their fragments using \eqref{eq:donor-signs}.
 \State Compute the retained cuts and proper-donor interfaces;
        remove nonessential retained fragments.
 \State Construct the new essential fragments using
        \eqref{eq:donor-reconstruct}--\eqref{eq:vertex-reuse}.
 \State Update supporting equations and incidences; reuse existing
        distances and update changed maximum records.
 \State Set $P\gets P\cup\{x\}$ and update the spatial index and scale.
 \State Output $x$ in reduced rational coordinates.
\EndWhile
\end{algorithmic}
\end{algorithm}
\FloatBarrier

\begin{theorem}[Correctness and operation count of LEFPS]\label{thm:lfpm}
For fixed dimension, rational polytopal representation, and rational initial point, Algorithm~\ref{alg:lfpm} generates exactly the point sequence \eqref{eq:canonical}. The sequence is rational and independent of the chosen cover of $\Omega$, and satisfies
\[
 \rho(P_n)\le2\qquad(n\ge2).
\]
The algorithm is autonomous. It also handles vertices at which more than $d$ constraint hyperplanes meet and Voronoi boundaries that coincide with input-cell facets. Generating the first $N$ points uses $O(N)$ local geometric computations, each involving a bounded number of rational operations. With a balanced comparison-based index for the maximum records, the total cost is $O(N\log(N+2))$ operations in the unit-cost model, and the storage is $O(N)$ scalar records.
\end{theorem}

\begin{proof}
Initialize the fragments and records from the fixed rational input. Lemma~\ref{lem:persistence} selects \eqref{eq:canonical}, and $h(P)>0$ ensures that the selected point is new. The spatial report contains the owner of every donor fragment, so Theorem~\ref{thm:reconstruct} restores the exact diagram. Induction gives rationality, cover independence, and autonomy; the mesh-ratio bound follows from \eqref{eq:greedy-separation}.

At grid scale $s$, both the old minimum pairwise distance and the insertion radius exceed $s/2$. Lemma~\ref{lem:compact} bounds each local update and the number of changed maximum records. Each update starts from a bounded-size constraint system and introduces only boundedly many additional inequalities, so removing redundant inequalities also requires bounded unit-cost work. Balanced-index operations cost $O(\log(N+2))$ per insertion; grid rebuilds cost $O(N)$ in total by \eqref{eq:grid-sum}. At most $N|\K|$ bounded-size fragments are live, giving $O(N)$ storage. Point output and one terminal state export add $O(N)$ unit-cost work.
\end{proof}

Reconstructing the entire diagram at every insertion creates at least $n$ fragments at $P_n$, so its total fragment-record work through $P_N$ is at least a constant multiple of $N^2$. This lower bound concerns repeated global reconstruction, not farthest-point sampling itself. Figure~\ref{fig:local-update} contrasts the two updates.

\begin{figure}[!htbp]
\centering
\includegraphics[width=\textwidth]{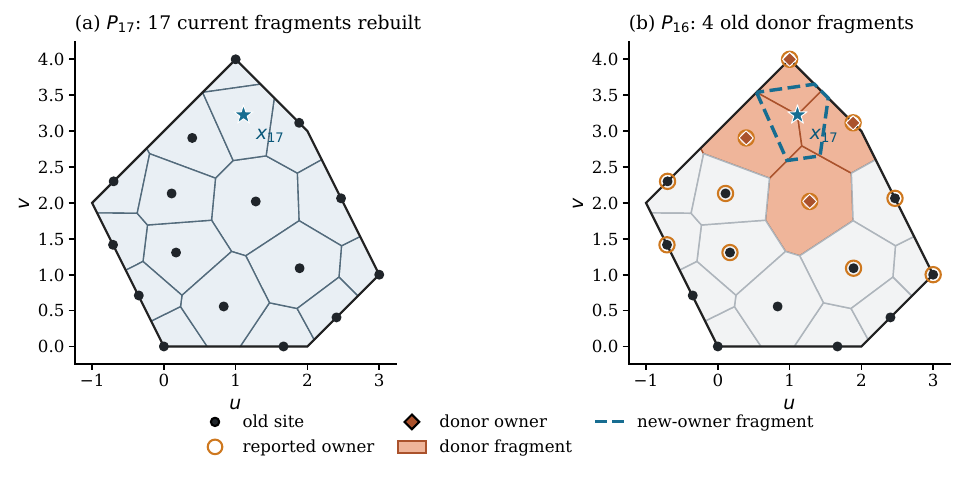}
\caption{A farthest-point insertion on the convex hexagon (H2).
Left: global reconstruction builds all 17 fragments of $P_{17}$.
Right: the local update identifies four donor fragments in $P_{16}$
and constructs the fragment owned by $x_{17}$. The spatial report
contains the owner of every donor fragment and may also contain other nearby owners.}
\label{fig:local-update}
\end{figure}

\section{Arithmetic-Aware Admissible Sampling}\label{sec:admissible}

The second algorithm, AAS, chooses short rational coordinates while preserving the mesh-ratio bound. We first characterize the allowed insertions and give an exact test using the current diagram. We then show that minimizing coordinate length over the entire candidate set requires testing only a bounded number of nearby candidates. Once we accept a point, we update the diagram as in Section~\ref{sec:geometry}.

\subsection{Insertions that preserve the mesh-ratio bound}
\label{subsec:admissible-condition}

Let $P$ be a singleton or a certified design. Recall that we call $x\in\Omega\setminus P$ \emph{admissible} if $\rho(P\cup\{x\})\le2$. The following theorem gives an equivalent condition in terms of the insertion radius and the new covering radius.

\begin{theorem}[Admissibility and insertion radii]\label{thm:admissible}
Suppose that $P$ is a singleton or $h(P)\le2q(P)$. For
$x\in\Omega\setminus P$, write
$r=d_P(x)$ and $H^+=h(P\cup\{x\})$. Then
\begin{equation}\label{eq:admissible}
 \rho(P\cup\{x\})\le2
 \quad\Longleftrightarrow\quad H^+\le r.
\end{equation}
Every sequence obtained from a singleton by admissible insertions
satisfies
\begin{equation}\label{eq:interlacing}
 2q(P_n) = \delta(P_n) = r_n \ge h_n \ge r_{n+1} \ge \frac{h_n}{2}\qquad(n\ge2).
\end{equation}
In addition, $r_2\ge h_1/2$ and
$h_n\asymp r_n\asymp n^{-1/d}$.
\end{theorem}

\begin{proof}
Since $r\le h(P)\le\delta(P)$, insertion makes the minimum pairwise distance equal to $r$, also for a singleton. Hence $\rho(P\cup\{x\})=2H^+/r$, proving \eqref{eq:admissible} and, by induction, $\delta(P_n)=r_n\ge h_n\ge r_{n+1}$.

Put $H=h(P)$. If $r<H$, an old farthest point $f$ must satisfy $\|f-x\|\le r$, since $\min\{H,\|f-x\|\}\le H^+\le r$. For a nearest site $p$ to $x$,
\[
 H=d_P(f)\le\|f-p\|\le\|f-x\|+\|x-p\|\le2r.
\]
The case $r=H$ is immediate. This also proves the bound for the first insertion. The asymptotic scales follow from \eqref{eq:volume-scales} and $h_n\le r_n=\delta(P_n)$.
\end{proof}

For an arbitrary finite starting design that is not assumed to be certified, the condition is instead $H^+\le\min\{2q(P),r\}$. Under the hypothesis of Theorem~\ref{thm:admissible}, every farthest point is admissible, because $H^+\le h(P)=r$. Other points may also be admissible; an example is given later in Example~\ref{ex:aai-choice}.

The equivalent condition in Theorem~\ref{thm:admissible} has a simple geometric interpretation: $x$ is admissible if and only if the closed balls of radius $r=d_P(x)$ centered at the points of $P\cup\{x\}$ cover $\Omega$:
\[
 \Omega\subseteq
 \left(\bigcup_{p\in P}B(p,r)\right)\cup B(x,r).
\]
Since the portion not covered by the balls centered at $P$ is
\[
 \Omega\setminus\bigcup_{p\in P}B(p,r)
 =\{y\in\Omega:d_P(y)>r\},
\]
admissibility is equivalent to 
\begin{equation}\label{eq:hole-set}
 \{y\in\Omega:d_P(y)>r\}\subseteq B(x,r).
\end{equation}
Thus, the ball centered at the new point must cover the portion the existing balls leave uncovered.

The lower bound $r_{n+1}\ge h_n/2$ also relates admissible insertion to relaxed greedy packing. The latter requires $r_{n+1}\ge\alpha_n h_n$ and guarantees $\rho\le2/a$ when $\alpha_n\ge a>0$; see \cite[Algorithm~2 and Theorem~3.6]{PZ}. Using only $a=1/2$ would give $\rho\le4$. The sharper bound $\rho\le2$ here follows from the additional post-insertion condition $H^+\le r$, not from the lower bound on the insertion radius alone.

\subsection{An exact covering test using vertices and edges}
\label{subsec:admissible-test}

We can test a candidate without constructing the diagram that would result from its insertion. On an old fragment $F=F_{\sigma,p}$, the new squared distance is $\min\{\|y-p\|^2,\|y-x\|^2\}$. Its maximum is attained at an old vertex or where the bisector cuts an old edge. We test these points against a squared threshold $t\ge0$; for admissibility, take $t=d_P(x)^2$.

For an edge $[u,v]$ of $F$, put
\[
 e=v-u,\qquad
 a=\|u-p\|^2-\|u-x\|^2,\qquad
 b=2e\cdot(x-p),
\]
and define
\begin{equation}\label{eq:quartic}
 Q_{F,e}(x,t)
  =(\|u-p\|^2-t)b^2-2(u-p)\cdot e\,ab+\|e\|^2a^2.
\end{equation}
Along the edge, $\ell_{x,p}(u+\theta e)=a+b\theta$. When the bisector crosses the edge interior, the sign of $Q_{F,e}$ determines whether the squared distance there exceeds $t$, without computing the intersection coordinates.

\begin{theorem}[Vertex--edge test for the new covering radius]
\label{thm:edge}
The inequality $h(P\cup\{x\})^2\le t$ holds if and only if every old essential fragment $F=F_{\sigma,p}(P)$ satisfies
\begin{align}
 &\|v-p\|^2\le t\quad\text{or}\quad\|v-x\|^2\le t
       &&\text{at every vertex }v,\label{eq:vertex}\\
 &a(a+b)\ge0\quad\text{or}\quad Q_{F,e}(x,t)\le0
       &&\text{at every edge }[u,v].\label{eq:edge}
\end{align}
For rational $x$ and $t$, these are rational sign tests and do not require the trial intersection coordinates. A fragment with $\mu_F^2\le t$ satisfies the conditions without further testing.

In particular, let $P$ be a singleton or certified design and let $t=r^2=d_P(x)^2<h(P)^2$. If $x$ is admissible, every fragment with $\mu_F^2>r^2$ has its owner in $B(x,h(P)+r)\subseteq B(x,2h(P))$.
\end{theorem}

\begin{proof}
Cut $F$ by $\Pi_{x,p}$. On each part, the new squared distance is a convex quadratic, so it suffices to test the vertices of that part, even when the part is lower-dimensional. A vertex created by the cut has $d-1$ independent old active constraints together with the cutting equation. Since it was not already a vertex of $F$, it lies in the interior of an old edge. This argument also applies when more than $d$ constraints are active at a vertex. Old vertices give \eqref{eq:vertex}.

An edge has a strict interior crossing exactly when $a(a+b)<0$. Then $b\ne0$, the crossing parameter is $-a/b$, and
\[
 b^2\bigl(\|u-(a/b)e-p\|^2-t\bigr)=Q_{F,e}(x,t).
\]
Its sign gives the required distance test. Endpoint crossings are already tested, and an edge in the bisector needs only its endpoints by convexity. The fragments cover $\Omega$, proving necessity and sufficiency.

If $\mu_F^2\le t$, the old owner already covers $F$ to the threshold. Otherwise, for an admissible $x$ and a maximizing vertex $v$ of a fragment with $\mu_F^2>r^2$, \eqref{eq:hole-set} gives $\|v-x\|\le r$. Thus $\|p-x\|\le\|p-v\|+\|v-x\|\le h(P)+r$.
\end{proof}

Call $F$ \emph{dangerous for $t$} if $\mu_F^2>t$: the old owner alone does not cover it to the required threshold. Only these fragments need testing. Old vertices alone do not suffice: a new maximum may occur at an edge--bisector intersection. For fixed geometry, these tests and the additional conditions $x\in\Omega$, $t>0$, and $\|x-p\|^2\ge t$ for all $p\in P$ can be expressed using finitely many polynomial inequalities of degree at most four in $(x,t)$, combined by ``and'' and ``or''.

\subsection{The candidate set and localization of the selection rule}
\label{subsec:aas-candidates}

Let $\mathcal V(P)$ and $\mathcal E(P)$ be the sets of geometric vertices and edges of all current essential fragments. We take the fragment vertices and edge midpoints as candidates, excluding points already in $P$:
\begin{equation}\label{aai:eq:pool}
 \mathcal C_{\AAS}(P)
   =\left(\mathcal V(P)
      \cup\{(u+v)/2:[u,v]\in\mathcal E(P)\}\right)\setminus P.
\end{equation}
Points with the same coordinates represent a single candidate, even if they arise from different fragments. Among these candidates, the admissible ones form the set
\[
 \mathcal A(P)
 =\{x\in\mathcal C_{\AAS}(P):
       h(P\cup\{x\})\le d_P(x)\}.
\]

Arithmetic-Aware Admissible Sampling (AAS) selects an admissible
candidate with the shortest coordinate representation, breaking
ties lexicographically:
\begin{equation}\label{aai:eq:full}
 x_{n+1}
 =\lexmin\operatorname*{arg\,min}_{x\in\mathcal A(P_n)}
       \lambda(x).
\end{equation}
The minimum is taken over all admissible candidates. This set is nonempty: every farthest point is a fragment vertex by Theorem~\ref{thm:witness}, and every such point is admissible. Unlike LEFPS, the AAS sequence may depend on the chosen polytopal cover, because the cover determines the fragment vertices and edges.

For the implementation, we order candidates by the pair $\kappa(x)=(\lambda(x),x)$: first by representation length, then lexicographically by coordinates. Both comparisons use the original input coordinates. We write $\min_\kappa S$ for the first point of a finite nonempty set $S$ in this order.

As shown below, the global rule \eqref{aai:eq:full} can be implemented locally because every admissible nonfarthest point must lie near any chosen farthest point.

\begin{lemma}[Localization around a farthest point]\label{lem:focal}
Let $P$ be a singleton or a certified design, let $H=h(P)$, and choose $f\in\Far(P)$. Every admissible point $x$ with $d_P(x)<H$ satisfies
\begin{equation}\label{eq:focal-locality}
 \|x-f\|\le d_P(x)<H.
\end{equation}
Only $O_{\K,d}(1)$ vertex and edge-midpoint candidate occurrences lie in $B(f,H)$, and their generating owners all lie in $B(f,2H)$.
\end{lemma}

\begin{proof}
At $f$, admissibility gives
$\min\{H,\|f-x\|\}\le d_P(x)<H$, proving \eqref{eq:focal-locality}.
A candidate occurrence $z$ belongs to its owner's fragment, so
$\|z-p\|=d_P(z)\le H$. Thus $z\in B(f,H)$ implies $p\in B(f,2H)$.
Lemma~\ref{lem:compact} bounds these owners and their vertices and edges
when $P$ is certified. A singleton has only the fixed input fragments.
\end{proof}

Use the farthest point with the smallest key as a reference point: 
\begin{equation}\label{aai:eq:focal}
 f(P)=\min_\kappa\Far(P).
\end{equation}
To maintain this reference point, choose each fragment's maximizing vertex by $\kappa$ and use $\kappa$ to break ties between equal fragment maxima. The first maximum record then gives $H^2$ and $f(P)$.

Since $f(P)$ is admissible, only candidates preceding it can be selected instead. Such candidates are nonfarthest, and any that is admissible lies in $B(f(P),H)$ by Lemma~\ref{lem:focal}. Retrieve owners in $B(f(P),2H)$ and generate their fragment vertices and edge midpoints. Retain those in $B(f(P),H)$ that precede $f(P)$, remove duplicates and current sites, and test the resulting bounded list in $\kappa$-order. Select the first admissible candidate, or $f(P)$ if none passes. This realizes \eqref{aai:eq:full}.

For a candidate $z$ generated by owner $p$,
$r_z^2=d_P(z)^2=\|z-p\|^2$. Thus its insertion radius is local, and
$z\in P$ exactly when $r_z^2=0$.

\subsection{Algorithm and computational complexity}
\label{subsec:aas-algorithm}

To test $z$, use the dyadic scale $U/2<H\le U$ and report owners including all of $B(z,2U)$. Let $\mathcal L$ contain their fragments. The report may include additional nearby owners but remains bounded. Reject $z$ if any fragment outside $\mathcal L$ has $\mu_F^2>r_z^2$: Theorem~\ref{thm:edge} would place every such owner within $2H$ for an admissible nonfarthest candidate. Otherwise apply the vertex--edge test to the dangerous fragments in $\mathcal L$; all other fragments already satisfy the threshold.

To check fragments outside $\mathcal L$, temporarily remove the records of $\mathcal L$ from the maximum index, compare the largest remaining value with $r_z^2$, and restore the removed records. An empty index returns no violation. A balanced index costs $O(\log(|P|+2))$ unit-cost operations; local sign tests cost bounded work. Lemma~\ref{lem:priority} gives the bit implementation. Candidate tests leave the diagram unchanged; reconstruction is needed only for the accepted point.

\begin{algorithm}[htbp]
\caption{Arithmetic-Aware Admissible Sampling (AAS)}\label{alg:aai}
\begin{algorithmic}[1]
\Require Rational polytopal cover $\K$ and initial point
         $x_1\in\Omega\cap\Q^d$.
\State Set $P\gets\{x_1\}$; initialize the essential fragments,
       $\kappa$-tied maximum records, and spatial index.
\State Output $x_1$ in reduced rational coordinates.
\While{another point is requested}
 \State Read $f=\min_\kappa\Far(P)$ and $H^2=h(P)^2$ from the
        first maximum record.
 \State Generate $\mathcal C_{\AAS}(P)\cap B(f,H)$ from nearby fragments.
 \State Remove duplicate coordinates and sort the candidates preceding
        $f$ by $\kappa$; set $x\gets f$.
 \For{each candidate $z$ in that list}
  \State Compute $r_z^2$ from a generating owner and collect nearby
         fragments in $\mathcal L$.
  \If{no fragment outside $\mathcal L$ has $\mu_F^2>r_z^2$}
   \If{all dangerous fragments in $\mathcal L$ pass the vertex--edge test}
    \State Set $x\gets z$ and leave the candidate loop.
   \EndIf
  \EndIf
 \EndFor
 \State Update the fragments for insertion of $x$ as in
        Section~\ref{sec:geometry}; restore current supports and
        $\kappa$-tied maximum records.
 \State Set $P\gets P\cup\{x\}$ and update the spatial index and scale.
 \State Output $x$ in reduced rational coordinates.
\EndWhile
\end{algorithmic}
\end{algorithm}
\FloatBarrier

\begin{theorem}[Correctness and operation count of AAS]\label{thm:aai}
For fixed dimension, rational polytopal representation, and rational initial point, Algorithm~\ref{alg:aai} generates exactly the sequence specified by \eqref{aai:eq:full}. It is deterministic and autonomous, produces rational points, and terminates at every insertion with $\rho(P_n)\le2$ for $n\ge2$. In particular,
\begin{equation}\label{aai:eq:short}
 \lambda(x_{n+1})
   =\min_{x\in\mathcal A(P_n)}\lambda(x)
   \le\min_{f\in\Far(P_n)}\lambda(f).
\end{equation}
The number of candidate tests per insertion is bounded independently of $n$. Through the first $N$ points, the algorithm uses $O(N)$ candidate tests and bounded-size local geometric computations. With a balanced comparison-based index, its total cost is $O(N\log(N+2))$ operations in the unit-cost model, and its storage is $O(N)$ scalar records.
\end{theorem}

\begin{proof}
The farthest fallback $f$ is a rational, admissible fragment vertex. Lemma~\ref{lem:focal} places every admissible candidate preceding $f$ in the local list. Theorem~\ref{thm:edge} and the check of fragments outside $\mathcal L$ test that list exactly. Thus the first candidate that passes, or $f$ if none passes, is precisely the point specified by \eqref{aai:eq:full}. This proves \eqref{aai:eq:short}. The exact update restores the fragments; rationality and the mesh bound follow inductively. The finite list and fallback ensure termination, and the rule is deterministic and autonomous.

At grid scale $U$, old sites are more than $U/2$ apart, and accepted insertions have radius at least $h(P)/2>U/4$. Lemma~\ref{lem:compact} bounds candidate generation, local tests, and fragment changes per insertion. Each maximum-index update or check of fragments outside $\mathcal L$ costs $O(\log(N+2))$ operations, while grid rebuilds cost $O(N)$ altogether by \eqref{eq:grid-sum}. There are $O(N)$ live records; temporary lists and obsolete geometry are discarded. Point output and one terminal export fit the same bounds.
\end{proof}

The minimization in \eqref{aai:eq:short} concerns the current admissible candidate set. It does not compare the cumulative output lengths of AAS and LEFPS, whose later designs and candidate sets can differ. Section~\ref{sec:arithmetic} bounds the bit cost of AAS in terms of its own generated coordinates, including work on rejected candidates and on reference farthest points that are not selected.

\begin{example}[An admissible point that is not farthest]\label{ex:aai-choice}
Consider the single-cell rectangle $\Omega=[0,2]\times[0,1]$ with $x_1=(0,0)$. LEFPS selects the unique farthest point $(2,1)$, whereas AAS selects $(1,1)$. The initial candidate set consists of the other three corners and the four edge midpoints. Using \eqref{eq:output-size}, the only candidates preceding $(1,1)$ in $\kappa$-order are $(0,1)$, $(1,0)$, and $(0,1/2)$. They fail the admissibility condition at $(2,0)$, $(2,1)$, and $(2,0)$, respectively. In each case, the indicated point is farther than the candidate's insertion radius from both $x_1$ and the candidate.

For $x=(1,1)$, the squared insertion radius is $r^2=2$. The bisector between $x_1$ and $x$ is $u+v=1$. On each of the two resulting polygons, the new squared nearest-site distance is convex. Evaluating it at the polygon vertices gives $h(\{x_1,x\})^2=2$, attained at $(2,0)$. Thus $(1,1)$ is admissible, although the old squared covering radius is $h(\{x_1\})^2=5$. Its representation is also shorter:
\[
 \lambda(1,1)=5<6=\lambda(2,1).
\]
Figure~\ref{fig:aas-choice} illustrates this first choice and the first six AAS points on the rectangle and a nonconvex cross. The subsequent sequences are determined in Subsection~\ref{subsec:rectangle-capture}.
\end{example}

\begin{figure}[!htbp]
\centering
\includegraphics[width=\textwidth]{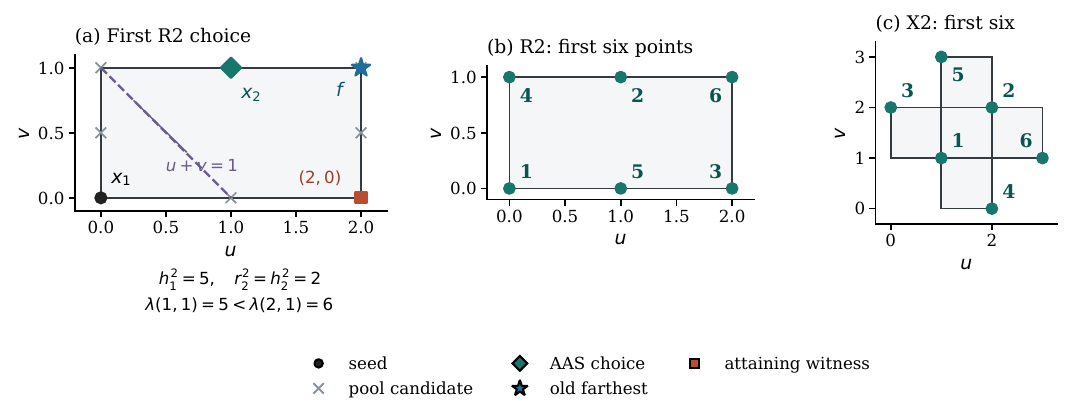}
\caption{AAS on the rectangle $[0,2]\times[0,1]$ (R2) and the
five-unit-square cross (X2). (a) From $x_1=(0,0)$, AAS selects the
nonfarthest point $x_2=(1,1)$, with $\lambda(1,1)=5$ rather than
$\lambda(2,1)=6$ for the old farthest point. The bisector $u+v=1$
and the point $(2,0)$ give $r_2^2=h(P_2)^2=2$, whereas $h(P_1)^2=5$.
(b)--(c) The first six AAS points, in insertion order, on the rectangle
and cross. The cross uses the five-square cover and initial point
$(1,1)$ in \eqref{eq:cross}.}
\label{fig:aas-choice}
\end{figure}

\section{Bit complexity and coordinate growth}\label{sec:arithmetic}

The operation counts in Sections~\ref{sec:geometry} and \ref{sec:admissible} treat rational arithmetic as having unit cost. We now account for the lengths of the numbers involved. We first prove that both algorithms can be executed in bit complexity nearly linear in their total output length $L_N$. This requires controlling both local arithmetic and comparisons with records elsewhere in the domain. We then turn to coordinate growth: we give a general denominator bound and prove that $L_N=\Theta(N\log(N+2))$ on rational interval covers.

\subsection{Arithmetic costs and explicit output}\label{subsec:bit-preliminaries}

We use the bit-complexity model of Subsection~\ref{subsec:input-output}. In particular, reading or writing $b$ bits at an $a$-bit address costs $O(a+b)$ bit operations. The cost of constructing a rational record includes reduction to lowest terms and computation of its length information. For the computations below on a bounded number of $b$-bit operands, fast multiplication~\cite{SchonhageStrassen}, continued-fraction computation~\cite{Schonhage}, greatest-common-divisor algorithms~\cite{Schonhage,Moller}, and fast integer division~\cite[Section~1.4]{BrentZimmermann} give the bound
\begin{equation}\label{eq:fastkernel}
 \mathfrak A(b)=O\bigl((b+1)\log^3(b+2)\bigr).
\end{equation}
The bit bounds use the ordered index of Subsection~\ref{subsec:priority}, rather than the balanced comparison-based indices used for the unit-cost bounds. Both implementations produce the same ordered points; a single implementation need not achieve both bounds when coordinate lengths are unrestricted.

\begin{lemma}[Lower bound for explicit output]\label{lem:encoding}
For $N$ distinct rational points in fixed dimension,
\begin{equation}\label{eq:entropy}
 N\log(N+2)=O_d(L_N),
 \qquad N\log(L_N+2)=O_d(L_N).
\end{equation}
Any algorithm that explicitly writes their reduced coordinates requires at least a constant multiple of $L_N$ bit operations.
\end{lemma}
\begin{proof}
Counting possible numerators and denominators gives at most $2^{2d(T+1)+O_d(1)}$ rational vectors of height at most $T$. For $T$ a sufficiently small positive multiple of $\log N$, this is less than $N/2$. Thus at least half of any $N$ distinct vectors have height greater than $T$, proving the first bound for large $N$; adjusting the constant covers the remaining values. For the second, write $L_N=Nt$, where $t\ge1$. The first bound gives $\log(N+2)=O_d(t)$, and hence
\[
 \log(L_N+2)\le\log(N+2)+\log(t+2)=O_d(t).
\]
The last claim follows from the number of bits written. It concerns explicit reduced coordinates, rather than symbolic formulas or shared expressions for several points.
\end{proof}

\subsection{Local operand lengths and repeated use of coordinates}
\label{subsec:reuse}

Let $S$ be the binary size of the preprocessed rational input, including the initial point and enclosing box, and put
\[
 \ell_n=1+\lceil\log_2(n+2)\rceil.
\]
A stored vertex may have been created many steps earlier, but its coordinate length can be bounded using its current defining equations. We call the set of sites occurring in these equations, including the owner when an owner distance is evaluated, a \emph{current support}. The supporting set may change; the site coordinates themselves do not. Using current equations for these bounds does not require recomputing the stored coordinates.

\begin{lemma}[Operand lengths from current supports]\label{lem:support}
Every current fragment vertex $v$ and its squared owner distance have a support $Q$ of bounded cardinality such that
\begin{equation}\label{eq:support-height}
 \height(v),\ \height(\|v-p\|^2)
 \le C_{\K,d}\left(S+1+\sum_{q\in Q}\height(q)\right).
\end{equation}
An AAS edge midpoint uses the union of the supports of its endpoints. Candidate radii, local sign tests, retained-cut and interface vertices, and feasibility tests satisfy bounds of the same form. Spatial grid coordinates require at most an additional $\ell_n$ term. For a vertex or candidate $v$ belonging to a current fragment, every supporting site $q$ satisfies $\|q-v\|\le C h(P)$ for a fixed $C$. For a trial intersection outside the fragment, the same bound holds with $v$ replaced by the fragment's owner.
\end{lemma}

\begin{proof}
A vertex solves $d$ independent active equations, each from an input facet or an owner bisector. Clear denominators and apply Cramer's rule. The resulting determinants have fixed degree in a bounded number of site coordinates, giving \eqref{eq:support-height}; reduction to lowest terms cannot increase these lengths. Midpoints, squared distances, and the tests of Section~\ref{sec:admissible} use boundedly many such operations. The same argument applies to trial intersections that fail the feasibility tests.

The local updates restore valid current supporting equations even when coordinates are reused. Lemma~\ref{lem:compact} places the sites in each constraint family within a fixed multiple of the covering radius of its owner. A feasible fragment point is also within that distance of its owner. Finally, $h_n\asymp n^{-1/d}$ bounds the binary lengths of grid addresses and grid endpoints by $O(S+\ell_n)$.
\end{proof}

At any fixed covering scale, packing limits how often a site can be used in local updates; see also \cite[Section~3.1]{HarPeledMendel}. An older site can participate at more scales. To account for this, define
\begin{align}
 w_{i,N}&=1+\log_2\frac{N+1}{i}, &
 W_N&=\sum_{i=1}^Nw_{i,N}\lambda_i,\label{eq:age-weight}\\
 \overline W_N&=\sum_{i=1}^Nw_{i,N}(\lambda_i+S+\ell_i).
 \label{eq:augmented-weight}
\end{align}
These quantities satisfy
\begin{equation}\label{eq:weight-bounds}
 L_N\le W_N\le(1+\log_2(N+1))L_N,
 \qquad \sum_{i=1}^Nw_{i,N}=O(N).
\end{equation}
Thus $\overline W_N=\widetilde O(L_N)$ on fixed input. Here and below, $\widetilde O$ suppresses fixed powers of logarithms of the point count and the total explicit output length, as in Section~\ref{sec:intro}.

For the implementations under consideration, a \emph{local computation} $\mathcal I$ is a numerical calculation or read of boundedly many nearby records. Its support $Q(\mathcal I)$ consists of the sites used to bound its operand lengths. Comparisons with distant records are treated separately in the next subsection.

\begin{theorem}[Repeated use of local coordinates]\label{thm:reuse}
For either LEFPS or AAS, let $m_N(x_i)$ be the number of local computations through $P_N$ whose support contains $x_i$. Then
\begin{equation}\label{eq:reuse}
 m_N(x_i)\le C_{\Omega,d,\K}\, w_{i,N}.
\end{equation}
Consequently, for arbitrary nonnegative weights $\omega_i$,
\begin{equation}\label{eq:weighted-reuse}
 \sum_{\mathcal I}\sum_{x_i\in Q(\mathcal I)}\omega_i
 \le C_{\Omega,d,\K}\sum_{i=1}^Nw_{i,N}\omega_i.
\end{equation}
The sum of all local operand lengths is $O(\overline W_N)$. This includes rejected candidates, farthest points used to guide AAS but not selected, and reads of stored numerical values.
\end{theorem}

\begin{proof}
Fix a period during which $U/2<h(P)\le U$ for one dyadic scale $U$. For LEFPS, every insertion that uses a site $p$ in a local computation lies within $AU$ of $p$, for a fixed $A$ (Lemma~\ref{lem:support}). The inserted points are more than $U/2$ apart, so only boundedly many can lie in this ball. Each insertion makes boundedly many local computations, and each grid rebuild reads a site once at that scale. In the final state export, each site supports only boundedly many records.

For AAS, let $f=f(P)$ be the reference farthest point and $x$ the accepted point. We have $\|x-f\|\le H\le U$. The report around $f$ lies in $B(f,C_0U)$, and any generated occurrence $z$ is within $H$ of its owner. Hence
\[
 \|z-x\|\le(C_0+2)U.
\]
This includes occurrences discarded before testing. Their supports, and the supports in subsequent local covering tests, are therefore within a fixed multiple of $U$ of the accepted point. Accepted points are $U/4$-separated and each step has bounded work, so the same packing bound covers rejected candidates and reference points that are not selected.

Since $h_n\asymp n^{-1/d}$, only $O(1+\log((N+1)/i))$ dyadic scales are visited after $x_i$ is inserted. This proves \eqref{eq:reuse}; double counting gives \eqref{eq:weighted-reuse}. Apply it with $\omega_i=\lambda_i$ and Lemma~\ref{lem:support}. Input data, scales, identifiers, and grid initialization add $O(N(S+\ell_N))$, absorbed by $\overline W_N$.
\end{proof}

\subsection{Exact ordering with variable-length keys}\label{subsec:priority}

A local update may compare a short rational number with a much longer one stored elsewhere. Repeatedly reading the longer number would escape the local reuse bound. We avoid this by encoding rational order as lexicographic order on binary strings using continued fractions~\cite{KornerupMatula}. A compressed prefix index~\cite{Morrison} then charges each operation to the key supplied by that update or query, even when stored keys are much longer. A code is \emph{prefix-free} if no code word is a proper prefix of another.

\begin{lemma}[Exact ordering charged to supplied-key lengths]
\label{lem:priority}
Let $z$ be a tuple containing a fixed number of rational numbers and nonnegative integer identifiers, ordered lexicographically with a prescribed direction for each entry. Stored tuples are immutable, and $b(z)$ denotes the total binary length of their entries. There is a prefix-free encoding that preserves this order, has length $O(b(z))$, and can be constructed in $O(\mathfrak A(b(z)))$ bit operations. With $a$-bit addresses, an index update costs $O((b(z)+1)a)$ after encoding; a minimum or maximum is returned by reference. If the queried scalar is the first tuple entry, testing whether any value outside a bounded collection $\mathcal L$ is strictly greater than $t$, or strictly less than $t$, costs
\begin{equation}\label{eq:priority-query}
 O\left(a\left[1+b(t)+\sum_{z\in\mathcal L}(b(z)+1)\right]\right)
\end{equation}
after encoding $t$. The tuple may include the representation length and coordinates used by the AAS key $\kappa$.
\end{lemma}

\begin{proof}
\emph{Order-preserving encoding.}
For $r\ge0$, use the canonical finite continued fraction $[a_0;\ldots,a_m]$, with $a_m\ge2$ when $m>0$. For a nonnegative integer $u$, set
\[
 U(u)=1^\ell0\operatorname{bin}_\ell(u+1),
 \qquad \ell=\bl(u+1),
\]
where $\operatorname{bin}_\ell$ denotes an $\ell$-bit binary word. The code $U$ is prefix-free and increasing: different lengths are ordered by the initial runs of ones, and equal lengths by the binary suffix. Its length is $2\ell+1$.

At position $j$, encode the coefficient $a_j$ as $0U(a_j)$.
Append a termination marker $1$ at position $m+1$, and complement each coefficient code or marker at an odd position (interchanging $0$ and $1$). Continued-fraction order alternates with the position; termination is the limit in which the next tail tends to $+\infty$~\cite{KornerupMatula}. The first unequal coefficient or marker therefore gives the numerical order, including when one expansion terminates first. Canonical expansions exclude a second representation ending in one. The completed codes are prefix-free. For signed rationals, prefix a nonnegative code by $1$ and a complemented positive-magnitude code by $0$ for a negative value. Encode nonnegative identifiers with $U$. Concatenate the codes for the tuple entries, complementing each code whose order is descending. This gives the desired tuple order.

\emph{Encoding length.}
For $r=p/q\ge0$, let $q_j$ be the continued-fraction denominator continuants. Their recurrence gives $q_m\ge\prod_{j=1}^m a_j$ and a Fibonacci lower bound on $q_m$ in terms of $m$. Thus
\[
 \sum_{j=1}^m\log_2a_j\le\log_2q,
 \qquad m=O(1+\log q).
\]
Since $\bl(a_j+1)\le\log_2a_j+2$ for $j\ge1$, these coefficient codes have total length $O(1+\log q)$. The integer part uses $O(1+\log(|p|+1))$ bits. Signs and the fixed number of tuple entries change only the constant. Fast continued-fraction computation gives the $O(\mathfrak A(b(z)))$ construction cost.

\emph{Index updates.}
Store the codes at the leaves of a compressed binary prefix index~\cite{Morrison}: branching nodes record the bit positions where codes differ, skipping shared runs of bits. Each node stores that position, its two child references, and references to its smallest and largest leaves. Branching positions strictly increase along a path. 

Number bit positions from zero. For a supplied $K$-bit word $w$, follow its bits while the branching position is less than $K$, then choose a representative leaf by reference. Prefix-freeness ensures that any distinct representative differs from $w$ at a position less than $K$. Finding the first difference therefore reads only $O(K)$ bits, even from a much longer representative. A second descent locates the insertion point. Repairing extreme references uses child references, without rereading remote keys. Deletion visits only branching positions below the deleted code's length and suppresses its parent; replacement is a deletion and insertion. Thus each update visits $O(K+1)$ nodes and bits.

With $a$-bit addresses, offsets, and branching metadata, the cost is $O((K+1)a)$ after encoding, or $O(\mathfrak A(b(z))+(b(z)+1)a)$ including encoding and transfers. An extreme is returned by reference; reading its contents is charged separately.

\emph{Queries excluding local records.}
Temporarily remove $\mathcal L$, keeping its codes for restoration. Removal and restoration cost $O(a\sum_{z\in\mathcal L}(b(z)+1))$. If no record remains, the answer is false. Otherwise take the appropriate extreme by reference and compare its scalar prefix with the code of $t$. Unequal scalar codes differ before the shorter ends; agreement throughout the threshold code implies equality by prefix-freeness. Thus the comparison reads only $O(1+b(t))$ bits and never needs the coordinate or identifier suffix of a long remote record. This proves \eqref{eq:priority-query}.
\end{proof}

Apply the lemma to the ascending fragment keys
\[
\begin{aligned}
 K_L(F)&=(-\mu_F^2,c_F,\operatorname{id}(\sigma),\operatorname{id}(p)),\\
 K_A(F)&=(-\mu_F^2,\lambda(c_F),c_F,
                 \operatorname{id}(\sigma),\operatorname{id}(p)).
\end{aligned}
\]
Here $c_F$ is the lexicographically smallest maximizing vertex for LEFPS, and the $\kappa$-smallest maximizing vertex for AAS. The identifiers distinguish stored occurrences without changing geometric ties. Only local candidate lists merge equal points. Both the old and new keys are charged when a record changes. Reading a selected point or the AAS reference point, and comparing locally generated candidates, are covered by Theorem~\ref{thm:reuse}.

The AAS test for fragments outside a local collection uses
\[
 \exists F\notin\mathcal L:\ \mu_F^2>t
 \quad\Longleftrightarrow\quad
 \min_{F\notin\mathcal L}(-\mu_F^2)<-t.
\]
An empty remaining index returns false. Excluded records keep their references and cached codes and are restored after the query. In the bit-complexity implementation, grid-scale comparisons are performed as scalar-prefix queries in the encoded index, using squared dyadic thresholds. The full value of a remote maximum is not read merely to decide whether the grid scale must be changed.

\subsection{Bit complexity in terms of the output length}
\label{subsec:output-sensitive-bound}

We now combine the local reuse bound with the ordered-index bound.

\begin{theorem}[Bit complexity relative to the output length]
\label{thm:output-sensitive}
Each of LEFPS and AAS admits an exact implementation satisfying
\begin{align}\label{eq:output-sensitive}
 T^{\rm bit}(N)
 &\le T^{\rm bit}_{\rm pre}
   +O_{\Omega,d,\K}\left(\overline W_N
                   \log^4(L_N+S+N+2)\right),\\
 S^{\rm bit}(N)&=\widetilde O_{\Omega,d,\K}(L_N+NS).\notag
\end{align}
Here $T^{\rm bit}_{\rm pre}$ is the preprocessing cost, $S^{\rm bit}(N)$ is the bit storage, and $L_N$ is the output length of the algorithm under consideration. In particular, on each fixed rational input, $T^{\rm bit}(N)=\widetilde O(L_N)$. The bounds include rejected trials, explicit reduced point output, and one export of the maintained state for continuation.
\end{theorem}

\begin{proof}
Put $Z_N=L_N+S+N+2$. Lemma~\ref{lem:support} and Theorem~\ref{thm:reuse} bound the total local operand length by $O(\overline W_N)$. Changed maxima, trial thresholds, and local comparisons use the same supports. Charging both old and new replacement keys, but no new encoding for unchanged maxima, gives the same bound for all supplied key lengths. Each operand has $O_I(Z_N)$ bits, so \eqref{eq:fastkernel} adds at most three logarithmic factors.

Lemma~\ref{lem:priority} charges index updates and exclusion queries to these supplied keys, local collections, and thresholds. Selected points and their supports are local to the insertion; Theorem~\ref{thm:reuse} also covers AAS reference points that are not selected. No uncharged complete remote key is read.

Packing bounds each site's multiplicity in the live geometry and keys. Their total numerical length is $O(L_N+NS+N\ell_N)$, with $O(N)$ index nodes; Lemma~\ref{lem:encoding} absorbs $N\ell_N$. Cumulative buffers are bounded by the operand and key sums, and numerical workspace by numerical work. Storage and bit offsets are therefore polynomial in $Z_N$, so addresses have $O(\log Z_N)$ bits. Index access and storage expansion by a fixed factor when capacity is exhausted fit within the fourth logarithmic factor. Reduced output and one pass over the live records, codes, links, and continuation data satisfy the same bounds and require no extra insertion.
\end{proof}

The theorem bounds execution cost in terms of the generated coordinates. In particular, near-linear bit complexity in $N$ is possible exactly when the total coordinate length is itself near-linear in $N$.

\begin{corollary}[When near-linear bit complexity is possible]
\label{cor:canonical-criterion}
For either specified LEFPS or AAS sequence on fixed input, an implementation using $N\log^{O(1)}(N+2)$ bit operations exists if and only if $L_N=N\log^{O(1)}(N+2)$.
\end{corollary}
\begin{proof}
Necessity follows from Lemma~\ref{lem:encoding}. Sufficiency follows from \eqref{eq:weight-bounds} and Theorem~\ref{thm:output-sensitive}.
\end{proof}

The local reuse bound alone does not establish this result for an implementation using ordinary rational comparisons: their cross products may repeatedly involve long keys stored elsewhere. The encoded index controls this additional cost.

\begin{corollary}[Bit complexity at logarithmic coordinate height]
\label{thm:log-heights}
If a specified LEFPS or AAS sequence satisfies $\height(x_i)=O(\log(i+2))$, its encoded-index implementation uses
\[
 O\bigl(N\ell_N^2+N\ell_N\log^3(\ell_N+2)\bigr)
       =O(N\log^2(N+2))
\]
bit operations, including rejected work and one final state export.
\end{corollary}

\begin{proof}
Every local operand and ordering key has $O(\ell_N)$ bits. There are $O(N)$ local computations, and the changed and query keys have total length $O(N\ell_N)$. Equation~\eqref{eq:fastkernel} bounds numerical work by $O(N\ell_N\log^3(\ell_N+2))$. Index accesses with $O(\ell_N)$-bit addresses cost $O(N\ell_N^2)$. Since $\log^3(\ell_N+2)=O(\ell_N)$ asymptotically, the stated order follows.
\end{proof}

For the explicitly described refinements in Section~\ref{sec:direct}, direct generation will improve this to $\Theta(N\log N)$ bit operations.

\subsection{Bounding denominator growth}\label{subsec:innovation}

The preceding results bound execution cost once $L_N$ is known. We now give a way to bound $L_N$ from formulas generating the points. Two effects matter: the depth of repeated operations and denominator factors not accounted for by the previous points.

Let $D_i$ be the least common multiple of the reduced coordinate denominators of $x_i$. Fix integers $Q\ge1$ and $q_0\ge2$, with $Q$ clearing the denominators of the fixed rational data used in the generation formulas. For each point, choose parent indices $\Pi_i\subset\{1,\ldots,i-1\}$ from its current supporting sites or an equivalent verified generation formula, and set
\begin{equation}\label{eq:innovation-data}
 U_i=\operatorname{lcm}\bigl(Q,(D_p)_{p\in\Pi_i}\bigr),\qquad
 c_i=\frac{D_i}{\gcd(D_i,q_0U_i)},\qquad
 \tau_i=1+\max_{p\in\Pi_i}\tau_p.
\end{equation}
The empty maximum is zero. The integer $\tau_i$ is the construction depth. The factor $c_i$, called the \emph{denominator innovation}, records additional prime factors or prime powers beyond those in $q_0U_i$. It depends on the chosen parents and $q_0$, not just on $x_i$. Let $\operatorname{Anc}(i)$ contain $i$ and all its distinct ancestors, and define
\[
 F_N(j)=|\{i\le N:j\in\operatorname{Anc}(i)\}|.
\]
Thus $F_N(j)$ counts the points that may inherit $c_j$, including $x_j$.

\begin{theorem}[Denominator growth through the parent relations]
\label{thm:innovation}
For the preceding data,
\begin{equation}\label{eq:innovation-divisibility}
 D_i\mid Qq_0^{\tau_i}\prod_{j\in\operatorname{Anc}(i)}c_j.
\end{equation}
If $\Omega\subset[-R,R]^d$ with $R\ge1$, then
\begin{equation}\label{eq:innovation-bound}
 L_N\le C_d\left[
 N\bigl(1+\log(R+1)+\log Q\bigr)
 +\log q_0\sum_{i\le N}\tau_i
 +\sum_{j\le N}F_N(j)\log c_j\right].
\end{equation}
\end{theorem}

\begin{proof}
Since $D_i=c_i\gcd(D_i,q_0U_i)$, we have $D_i\mid c_iq_0U_i$. Induct on the parent relations. A common multiple of all parent denominators uses $Q$ once, the largest parent depth as the exponent of $q_0$, and the product of the innovations over the union of the parent ancestor sets. Multiplication by $q_0c_i$ proves \eqref{eq:innovation-divisibility}; a shared ancestor need not be counted more than once.

Every reduced coordinate denominator of $x_i$ divides $D_i$, and every reduced numerator has absolute value at most $RD_i$. Consequently,
\[
 \lambda(x_i)\le C_d(1+\log(R+1)+\log D_i).
\]
Take logarithms in \eqref{eq:innovation-divisibility}, sum over $i$, and interchange the ancestor sums to obtain \eqref{eq:innovation-bound}.
\end{proof}

\begin{corollary}[A sufficient condition for short coordinates]
\label{cor:innovation}
On fixed input and for fixed $q_0$, suppose that both $\sum_{i\le N}\tau_i$ and $\sum_{j\le N}F_N(j)\log c_j$ are $O(N\log^c(N+2))$ for some fixed $c\ge1$. Then $L_N=O(N\log^c(N+2))$, and the specified LEFPS or AAS sequence admits an implementation with $N\log^{O(1)}(N+2)$ bit operations. For $c=1$, one has $L_N=\Theta(N\log(N+2))$.
\end{corollary}

\begin{proof}
Apply Theorem~\ref{thm:innovation}, Theorem~\ref{thm:output-sensitive}, and Lemma~\ref{lem:encoding}.
\end{proof}

If each point is given by an affine formula in its parent coordinates, with coefficient denominators dividing $q_0$ and constant-term denominators dividing $q_0Q$, then $D_i\mid q_0U_i$ and $c_i=1$. Logarithmic construction depth then gives logarithmic coordinate height; midpoint refinement uses $q_0=2$. A finite set of nonlinear formulas need not suffice: repeated squaring of $1/2$ gives denominator length $2^k+1$ after $k$ steps. The ancestor counts here describe propagation through generation formulas; the reuse counts in Theorem~\ref{thm:reuse} describe reads during execution. These are different quantities, and cancellations may make the denominator bound conservative.

\subsection{Coordinate growth on rational interval domains}
\label{subsec:interval-arithmetic}

On intervals, every admissible choice from the vertex--midpoint pool has logarithmic height, even without minimizing coordinate length.

\begin{theorem}[Coordinate growth on rational interval covers]
\label{thm:interval-arithmetic}
Let $\K$ be a fixed finite family of nondegenerate compact rational intervals, with arbitrary overlaps and possibly disconnected union. Starting from any fixed rational singleton, every sequence obtained by choosing admissible members of \eqref{aai:eq:pool} satisfies
\begin{equation}\label{eq:interval-output}
 \height(x_n)=O_I(\log(n+2)),\qquad
 L_N=\Theta_I(N\log(N+2)).
\end{equation}
Only $O_I(\log(N+2))$ of the first $N$ insertions are nonfarthest. The constants are uniform over these admissible choices on the fixed input. In particular, both LEFPS and AAS have these properties and admit maintained exact implementations with $O_I(N\log^2(N+2))$ bit complexity.
\end{theorem}

\begin{proof}
Let $E$ be the set of all endpoints in the original cover, and let $D$ be the length of an enclosing interval. Since $h_n\le\delta(P_n)\le D/(n-1)$, after a number of insertions depending only on the input, each component has at least two sites and its nearest-site problem is independent of the other components. Because each step has finitely many candidates, there are only finitely many histories up to this fixed index. Their coordinate heights can therefore be absorbed into an input-dependent constant.

\emph{Exceptional insertions are close to cover endpoints.}
Call a later insertion \emph{ordinary} if it is the midpoint of two consecutive sites in one component. A fragment vertex is either such a midpoint or a point of $E$. A fragment with no endpoint in $E$ is a complete interior Voronoi interval. If its owner $p$ has neighboring gaps of lengths $u,v$, the midpoint of this fragment is $x=p+(v-u)/4$. For $x>p$, its insertion radius is $r=(v-u)/4<v/4$, whereas the remaining gap on the right gives $h(P\cup\{x\})\ge(v-r)/2>r$. It is therefore inadmissible. The case $x<p$ is symmetric, and $x=p$ is not a candidate. Call a nonordinary insertion \emph{exceptional}. Such an admissible insertion is either an input endpoint or the midpoint of a fragment ending at one. As a fragment has length at most $2H$, where $H=h(P)$, every such insertion lies in
\begin{equation}\label{eq:interval-exception}
 \bigcup_{e\in E}[e-H,e+H].
\end{equation}

Within $U/2<H\le U$, accepted points are more than $U/4$ apart. There are consequently at most $9|E|$ exceptional insertions at this scale. The covering bound $h_n\ge\vol(\Omega)/(2n)$ permits only $O_I(\log(N+2))$ scales through $P_N$. An ordinary midpoint has insertion radius equal to half its gap length. It cannot change nearest-site distances outside that gap. If it were not farthest, an old farthest point outside the gap would therefore remain at a distance greater than the insertion radius, violating admissibility. This proves the nonfarthest count.

\emph{The affine construction has logarithmic depth.}
Let $r_p$ be the insertion radius when the site $p$ was created. An ordinary point $p$ begins with adjacent gaps of length $r_p$, and those gaps only shrink. A later ordinary midpoint $x$ using $p$ as a parent therefore satisfies
\begin{equation}\label{eq:ordinary-contraction}
 r_x\le r_p/2.
\end{equation}
Use the two adjacent sites as parents of an ordinary point. Treat input endpoints as fixed data. For an exceptional fragment midpoint, express each endpoint as either fixed data or the midpoint of two old sites, and take the sites in this formula as parents.

Insertion radii at one scale lie in $(U/4,U]$. By \eqref{eq:ordinary-contraction}, a chain of ordinary parent--child relations has bounded length within that scale. There are only boundedly many exceptional insertions at the same scale, so they split any parent path into boundedly many such chains. The finite initial histories add only a constant to the depth. Thus each parent path has $O_I(1)$ nodes per scale and $\tau_i=O_I(\log(i+2))$.

All coefficients in these affine formulas have denominators dividing four. Let $Q$ be a common multiple of the reduced denominators of all endpoints in $E$ and of the initial point. Counting construction depth from these fixed rational data, induction gives
\[
 \operatorname{den}(x_i)\mid Q 4^{\tau_i},
\]
where $\operatorname{den}$ denotes the positive reduced denominator. The bounded domain gives the same logarithmic bound for numerator lengths. Summation, Lemma~\ref{lem:encoding}, and Corollary~\ref{thm:log-heights} complete the proof.
\end{proof}

The candidate-set restriction is important. The mesh-ratio bound alone cannot control coordinate precision, even on the unit interval.

\begin{proposition}[The mesh ratio does not bound coordinate height]
\label{prop:long-dyadic}
For every computable $F:\mathbb N\to\mathbb N$, there is a deterministic admissible sequence of dyadic points on $[0,1]$, starting with $x_1=0$ and $x_2=1$, such that $\height(x_n)\ge F(n)$ for every $n\ge3$. This assertion concerns unrestricted admissible selection, not LEFPS or AAS.
\end{proposition}

\begin{proof}
Maintain distinct lengths for all adjacent gaps. For the unique longest gap $[a,a+L]$, let $L_2<L$ be the second-longest length, or zero if there is no other gap. Choose
\[
 x=a+L/2+\varepsilon,\qquad
 0<\varepsilon<\min\{(L-L_2)/2,L/6\}.
\]
Then $r=d_P(x)=L/2-\varepsilon$ and
\[
 h(P\cup\{x\})=\max\{L_2/2,L/4+\varepsilon/2\}<r.
\]
Remove the finitely many choices for which a new gap has the same length as an old gap. The two new gaps are distinct because $\varepsilon>0$. The remaining nonempty open set contains dyadic points with exact denominator $2^B$ for arbitrarily large $B$: odd numerators over this denominator have spacing $2^{1-B}$. At step $n$, search increasing $B\ge\max\{1,F(n)\}$ and take the least permitted odd numerator at the first feasible exponent. Exact interval and equality tests make this a deterministic terminating procedure. Admissibility preserves the mesh-ratio bound, and the reduced denominator has $B+1$ bits.
\end{proof}

\section{Direct generation of structured point sequences}\label{sec:direct}

On certain inputs, the ordered LEFPS or AAS sequence has an explicit refinement description. Generating these points directly avoids Voronoi maintenance. We call this \emph{Direct Stream Execution} (DSE). It reproduces the same points in the same order, outputs exact $h_n^2$ and $\delta(P_n)^2$ for $n\ge2$, and may stop after any $P_N$. All preparation is charged when performed, including work on points never requested. The final export contains the generator's continuation state, rather than a Voronoi diagram.

\subsection{Cost of direct refinement}
\label{subsec:direct-criterion}\label{subsec:direct-mechanisms}

The constructions below use dyadic refinements of fixed rational data. The following criterion includes preparation at incomplete levels.

\begin{proposition}[Cost of dyadic direct generation]
\label{thm:ldcc}\label{thm:direct-refinement}
Suppose an autonomous direct construction reproduces a specified sequence after a fixed initial part. Level $k$ starts with $\Theta_I(2^{kd})$ points and has a bounded number of phases. Assume each phase requires $O_I(2^{kd})$ preparation operations and bounded work per output, including exact quality measures. Preparation begins only when another point is requested. Suppose numerical data have the form $z/(D2^e)$, with fixed $D$, $\bl(z),e=O_I(k+1)$, and are formed by additions, shifts, and multiplications by fixed integers. Ordering may use a fixed number of $O_I(k+1)$-bit integer keys. If the live continuation state has $O_I(N)$ records of this width, then, at every stopping point,
\begin{equation}\label{eq:direct-cost}
 T^{\rm str}(N)=O_I(N),\qquad
 T^{\rm bit}(N)=\Theta_I(N\log(N+2)),\qquad
 L_N=\Theta_I(N\log(N+2)).
\end{equation}
These bounds include the initial part and one final state export.
\end{proposition}

\begin{proof}
Preparation through the last level started sums geometrically to $O_I(2^{kd})$. That many points have already been generated at the start of the level, so this is $O_I(N)$ even if generation stops immediately after preparation. Output work is also $O_I(N)$; this is aggregate amortization in the sense of~\cite{Tarjan}.

Reduction and ordering have linear bit cost at the stated width. Indeed, write $D=2^uD_o$ with $D_o$ odd. For $z\ne0$, put $g=\gcd(|z|,D_o)$ and $t=\min\{v_2(|z|),e+u\}$; the reduced pair is
\[
 \left(z/(g2^t),\ (D_o/g)2^{e+u-t}\right).
\]
Computing the remainder modulo the fixed integer $D_o$, counting trailing zeros, dividing by a fixed integer, and performing bit shifts require $O_I(k+1)$ bit operations in total. The coordinate zero is represented as $0/1$. This also computes the exact reduced lengths used by $\kappa$. For $M=O_I(2^{kd})$ records, stable radix sorting with a power-of-two radix comparable to $2^{kd}$ takes a bounded number of counting passes. It costs $O_I(2^{kd})$ operations and $O_I(2^{kd}(k+1))$ bit operations, including counters, addresses, and transfers. With a common positive denominator, coordinate order is numerator order.

Since $k=O_I(\log(N+2))$, arithmetic, data access, and final export cost $O_I(N\log(N+2))$ in total. The coordinate forms give the same upper bound for $L_N$; Lemma~\ref{lem:encoding} gives optimality.
\end{proof}

In the planar and four-dimensional refinements, the covering radius
is constant during a partial phase and changes at its last point. A prepared phase count identifies that point without looking ahead to the next phase. The insertion radius $r$ updates the minimum distance by $\delta(P\cup\{x\})^2=\min\{\delta(P)^2,r^2\}$. Thus exact quality requires only fixed formulas and a counter. The preparation bound is amortized: the first request in a phase may cost more than subsequent requests.

\subsection{LEFPS on rational interval unions}
\label{subsec:certificate-instances}

\begin{corollary}[Direct generation on interval unions]
\label{thm:interval-writer}
On a fixed rational interval union, LEFPS from a rational singleton or a fixed certified rational initial design eventually selects every component endpoint. Its ordered continuation admits the bounds \eqref{eq:direct-cost}, including the initial part and exact quality.
\end{corollary}
\begin{proof}
Since $h_n\to0$, the nearest-site problems of distinct connected components eventually become independent. From that time, a missing component endpoint is the unique maximizer on the interval between it and the nearest site in its component. No other farthest-point insertion can enter and shorten that interval. If the endpoint were never selected, its distance from the design would therefore remain positive, contradicting $h_n\to0$.

After all component endpoints have been selected, fix the finitely many gaps $[a_\alpha,b_\alpha]$ between consecutive sites within each connected component, of lengths $l_\alpha$. These root gaps cover their respective components and have no sites in their interiors. Within root gap $\alpha$, successive midpoint levels are
\[
 a_\alpha+\frac{(2t+1)l_\alpha}{2^{j+1}},\quad
 0\le t<2^j,\qquad R_{\alpha,j}=\frac{l_\alpha}{2^{j+1}}.
\]
Enumerate each level in increasing $t$. Keep the next point and insertion radius of each root, select the largest radius, break radius ties by the smaller coordinate, and advance that root alone. This is LEFPS: unbisected gaps are longer than their children, and the selected point is the leftmost remaining midpoint at the largest radius. The largest updated root radius gives $h_n$, and the selected radius updates $\delta(P_n)$.

Reaching level $j$ requires $2^j-1$ earlier points in that root.
Thus all data have $O_I(\log(N+2))$ bits. There are only fixedly many roots, and their comparisons and output use shifts and fixed rational coefficients. Each request therefore costs bounded unit-cost work and $O_I(\log(N+2))$ bit work. Lemma~\ref{lem:encoding} gives the matching lower bound.
\end{proof}

\subsection{AAS on a rectangle and a nonconvex cross}
\label{subsec:rectangle-capture}

We consider the single-cell rectangle
\begin{equation}\label{eq:rectangle-input}
 \Omega_R=[0,2]\times[0,1],\qquad x_1=(0,0),
\end{equation}
and the cross with its five-square input cover
\begin{equation}\label{eq:cross}
 \begin{gathered}
 \Omega_+=\bigcup_{(i,j)\in C}[i,i+1]\times[j,j+1],\qquad x_1=(1,1),\\
 C=\{(0,1),(1,0),(1,1),(1,2),(2,1)\}.
 \end{gathered}
\end{equation}
For either domain, let $G_k$ be the grid vertices at spacing $s=2^{-k}$, and let $A_k,E_k$ be the centers and side midpoints of the fine squares, with shared points counted once. Then $G_{k+1}=G_k\mathbin{\dot\cup}A_k\mathbin{\dot\cup}E_k$. These squares describe the refinement; they do not replace the original input cells in the AAS candidate set.

The geometric pattern is the square refinement of Pronzato and Zhigljavsky~\cite[Theorem~3.3]{PZ}. The additional issue is agreement with AAS's coordinate-length ordering.

\begin{lemma}[Square refinement and AAS ordering]
\label{lem:rect-geometry}\label{lem:planar-phases}\label{lem:rect-barrier}
For proper subsets $A'\subsetneq A_k$ and $E'\subsetneq E_k$,
\begin{align}
 h(G_k\cup A')^2&=s^2/2,&
 \Far(G_k\cup A')&=A_k\setminus A',\label{eq:rect-center-phase}\\
 h(G_k\cup A_k\cup E')^2&=s^2/4,&
 \Far(G_k\cup A_k\cup E')&=E_k\setminus E'.\label{eq:rect-edge-phase}
\end{align}
The completed phases have squared covering radii $s^2/4$ and $s^2/8$. On either input, every admissible nonfarthest candidate has greater length than a remaining farthest point. For the rectangle, the same length comparison holds if $P\subsetneq G_k$ meets every actual grid edge; then $h(P)=s$ and $\Far(P)=G_k\setminus P$.
\end{lemma}

\begin{proof}
For vertices $v_i$ on a circle of center $c$ and radius $R$, and $y=\sum_i\alpha_i v_i$ with convex coefficients,
\begin{equation}\label{eq:planar-barycentric}
 \sum_i\alpha_i\|y-v_i\|^2=R^2-\|y-c\|^2.
\end{equation}
Apply this first to each square, then to the four triangles joining its center to its sides. Their circumradii are $s/\sqrt2$ and $s/2$, and equality can occur only at the square centers and side midpoints, respectively. Distinct points in each phase are separated by at least its covering radius, so uninserted ones remain exactly the farthest points. The argument covers the actual domain, including the cross boundary. A set meeting every square edge contains two opposite corners, which cover that square to distance $s$, with equality only at missing corners.

For the length comparison, Lemma~\ref{lem:focal} confines an admissible nonfarthest candidate to a neighborhood of a remaining farthest point. The original fragment edges then permit only finer dyadic perturbations, which strictly increase reduced coordinate length. The boundary cases and complete comparison are proved in Appendix~\ref{app:planar-comparison}.
\end{proof}

\begin{theorem}[The AAS sequences on the rectangle and cross]
\label{thm:rectangle-aai}\label{thm:cross}
For the rectangle \eqref{eq:rectangle-input}, the initial points are
\begin{equation}\label{eq:rect-start}
 (0,0),\ (1,1),\ (2,0),\ (0,1),\ (1,0),\ (2,1),
\end{equation}
giving $P_6=G_0$. For the cross \eqref{eq:cross}, they are
\begin{equation}\label{eq:crossstart}
 \begin{gathered}
 (1,1),\ (2,2),\ (0,2),\ (2,0),\ (1,3),\ (3,1),\\
 (0,1),\ (1,0),\ (1,2),\ (2,1),\ (2,3),\ (3,2),
 \end{gathered}
\end{equation}
giving $P_{12}=G_0$. Thereafter AAS inserts $A_k$ and then $E_k$, each in increasing $\kappa$-order. Only $x_2$ is nonfarthest, and $\rho(P_n)\in\{\sqrt2,2\}$ for all $n\ge2$. Both sequences have logarithmic coordinate heights and admit \eqref{eq:direct-cost}, including exact quality. On the rectangle, more precisely,
\begin{equation}\label{eq:rect-bits}
 P_{(2^{k+1}+1)(2^k+1)}=G_k,\qquad
 \height(x_n)\le2+\tfrac12\log_2n\quad(n\ge2).
\end{equation}
\end{theorem}

\begin{proof}
The finite initial choices and their quality measures are verified in Appendix~\ref{app:planar-entry}; Example~\ref{ex:aai-choice} already gives the rectangle's second point. The preceding lemma then forces the $\kappa$-smallest remaining farthest point at every step, proving the phase descriptions. The covering radius equals the latest insertion radius except when phase completion reduces it by $\sqrt2$; the initial data have the same alternatives. This gives the mesh ratios.

On the rectangle, $|G_k|=(2^{k+1}+1)(2^k+1)$; on the cross,
\begin{equation}\label{eq:crosscounts}
 |G_k|=5\cdot4^k+6\cdot2^k+1,\quad
 |A_k|=5\cdot4^k,\quad |E_k|=10\cdot4^k+6\cdot2^k.
\end{equation}
Each phase is prepared by scanning $O(4^k)$ grid addresses, removing shared occurrences and sorting by $\kappa$. Coordinates have common denominator $2^{k+1}$ and $O(k+1)$-bit numerators; phase counts give exact quality. Proposition~\ref{thm:direct-refinement} applies. For the sharper rectangle height bound, phase-$k$ numerators have at most $k+2$ bits and the insertion index exceeds $|G_k|\ge2\cdot4^k$; the initial points satisfy the bound directly.
\end{proof}

\subsection{LEFPS on four-dimensional coordinate-order domains}
\label{subsec:four}

For a nonempty set $\Pi$ of permutations of $\{1,2,3,4\}$, put
\begin{equation}\label{eq:orderdomain}
 \Delta_\pi=\{x\in[0,1]^4:x_{\pi(1)}\le\cdots\le x_{\pi(4)}\},
 \qquad \Omega_\Pi=\bigcup_{\pi\in\Pi}\Delta_\pi.
\end{equation}
These unions include the four-dimensional order polytopes~\cite{Stanley} and may be nonconvex. The four-cube refinement of~\cite[Theorem~3.5]{PZ} extends to them because a nearest lattice point can be moved into an order chamber without increasing distance.

Write
\[
 D_4=\{z\in\Z^4:\textstyle\sum_i z_i\text{ is even}\},\qquad
 D_4^*=\Z^4\cup\bigl(\Z^4+(1/2,1/2,1/2,1/2)\bigr).
\]
For a lattice $\Lambda$, denote its covering radius by $R(\Lambda)$ and its set of farthest points, or holes, by $\mathcal H(\Lambda)$.

\begin{lemma}[Lattice refinement at the domain boundary]\label{lem:fold}
\begin{align}\label{eq:latticeholes}
 R(D_4)&=1,&\mathcal H(D_4)&=D_4^*\setminus D_4,\\
 R(D_4^*)&=1/\sqrt2,&
 \mathcal H(D_4^*)&=(\tfrac12D_4)\setminus D_4^*.\notag
\end{align}
Distinct holes are separated by at least the corresponding covering radius. For either lattice, every integer $Q\ge1$ and $y\in Q\Omega_\Pi$,
\begin{equation}\label{eq:fold}
 \dist(y,\Lambda\cap Q\Omega_\Pi)=\dist(y,\Lambda).
\end{equation}
\end{lemma}

\begin{proof}
Round each coordinate to the nearest integer. If the sum of the rounded coordinates is even, the resulting point belongs to $D_4$ and has squared distance at most one. Otherwise, let $t\in[0,1/2]$ be the largest absolute rounding error. Changing the corresponding rounded coordinate by one in the direction of the original coordinate gives a point in $D_4$ at squared distance at most $4t^2+1-2t\le1$. If $t=0$, either direction may be used. Equality occurs precisely at integer vectors whose coordinate sum is odd, or at vectors all of whose coordinates belong to $\mathbb Z+1/2$. These are exactly the points of $D_4^*\setminus D_4$. For $D_4^*$, let $a_i\in[0,1/2]$ be the absolute distance from the $i$th coordinate to a nearest integer. The squared distances to the two cosets are
\[
 A=\sum_{i=1}^4 a_i^2,
 \qquad
 B=\sum_{i=1}^4(1/2-a_i)^2.
\]
Since $\min\{A,B\}\le (A+B)/2\le 1/2$, the covering radius is at most $1/\sqrt{2}$. Equality holds precisely when $A=B$ and $a_i\in\{0,1/2\}$ for every $i$. Thus, exactly two coordinates are integers, and the other two are half-integers. The asserted hole separations follow from these coordinate descriptions.

For \eqref{eq:fold}, take a nearest lattice point to $y$. Repeatedly reflect any coordinate outside $[0,Q]$ across the corresponding boundary plane at zero or $Q$ until the point lies in $[0,Q]^4$. Then sort its coordinates into the order of a chamber containing $y$. Both lattices are preserved. Reflections do not increase distance; swapping an inverted pair decreases squared distance by $2(y_j-y_i)(z_i-z_j)\ge0$ when $y_i\le y_j$ and $z_i>z_j$. The resulting point lies in the required chamber.
\end{proof}

\begin{theorem}[The LEFPS sequence on coordinate-order domains]
\label{thm:four}
From $x_1=0$, LEFPS on \eqref{eq:orderdomain} inserts $\mathbf1$, then all binary vectors in the domain with exactly two ones, in lexicographic order. This reaches $A_0$ in at most eight points, where
\begin{equation}\label{eq:fourphases}
 A_k=(2^{-k}D_4)\cap\Omega_\Pi,\qquad
 B_k=(2^{-k}D_4^*)\cap\Omega_\Pi.
\end{equation}
It then inserts $B_k\setminus A_k$ followed by $A_{k+1}\setminus B_k$, each in lexicographic order. At complete phases, $h(A_k)^2=4^{-k}$ and $h(B_k)^2=4^{-k}/2$; during a proper partial phase the covering radius remains at its starting value. Every nonsingleton design has mesh ratio in $\{\sqrt2,2\}$, and coordinate heights are logarithmic. Direct generation satisfies \eqref{eq:direct-cost}, including exact quality.
\end{theorem}

\begin{proof}
The unique farthest point from zero is $\mathbf1$. For $P=\{0,\mathbf1\}$, putting $s=\sum_i x_i$ gives $d_P(x)^2\le s\le2$ when $s\le2$; use $\mathbf1-x$ otherwise. Equality holds exactly at binary vectors with two ones, which occur in every chamber and are mutually at least $\sqrt2$ apart. This proves the initial ordering.

Lemma~\ref{lem:fold} identifies the holes at each subsequent phase. Each phase is nonempty in every chamber, and its holes are mutually at least the old covering radius apart. Thus, inserting one leaves all remaining holes farthest, while no other point can become a maximizer at that radius. This proves the phase order and radii. Equation~\eqref{eq:greedy-separation} gives the mesh ratios.

At scale $Q=2^k$, scan addresses $u\in[0,2Q]^4\cap\Z^4$ and output $u/(2Q)$ when it lies in the domain and has the phase's parity: four odd coordinates, or four even coordinates with $\sum_i u_i\equiv2\pmod4$, for the first phase; exactly two odd coordinates for the second. Lexicographic scans avoid duplicates between chambers. There are $O(Q^4)$ visits, while coverage and grid containment give $|A_k|=\Theta_\Pi(Q^4)$. Prepare each phase list and its count on its first request. Its count then identifies the final point and the change in quality without another scan. Proposition~\ref{thm:direct-refinement} applies.
\end{proof}

\begin{corollary}[Rational similarity images]\label{cor:fourframe}
The preceding result extends to $a+B\Omega_\Pi$ from $a$, for rational $a,B$ with $B^{\mathsf T}B=\sigma^2I$, $\sigma^2>0$, using lexicographic ties in the image coordinates.
\end{corollary}

\begin{proof}
Squared distances scale by $\sigma^2$. Enumerate the same phase sets, form image coordinates over a common denominator, and radix-sort their numerators. Fixed rational coefficients preserve the arithmetic bounds of Proposition~\ref{thm:direct-refinement}.
\end{proof}

The numerical example Q4 in the next section uses the single chamber
\begin{equation}\label{eq:simplex}
 \Delta=\{0\le x_4\le x_3\le x_2\le x_1\le1\}.
\end{equation}
General affine shears need not preserve \eqref{eq:fold}.

\subsection{Independent components and the limitation of AAS scaling}
\label{subsec:canonical-merge}\label{subsec:scaling-obstruction}

\begin{proposition}[Merging independent LEFPS constructions]
\label{prop:direct-merge}
Let $\Omega=\bigcup_{i=1}^r\Omega_i$ be a fixed finite union, and let $P_0$ be a fixed already generated certified design or singleton. Suppose each $P_0\cap\Omega_i$ is nonempty, $\Omega_i\cap\Omega_j\subseteq P_0$ for $i\ne j$, and throughout continuation
\begin{equation}\label{eq:merge-independence}
 d_P(y)=d_{P\cap\Omega_i}(y)\qquad(y\in\Omega_i).
\end{equation}
Assume that each local generator produces the LEFPS continuation on $\Omega_i$ from $P_0\cap\Omega_i$, with all distance ties broken lexicographically in the common ambient coordinates. Each local generator must satisfy \eqref{eq:direct-cost}, including all work needed to make its current squared covering radius and its next point available at every local stopping prefix, without inserting that point, and to refresh this pair after an insertion. Assume also that comparing these radii and next points costs $O_I(1)$ unit-cost and $O_I(\log(N+2))$ bit operations. Selecting the largest squared local radius, with lexicographic next-point ties, and advancing only that generator then gives the global LEFPS continuation from $P_0$ with the same bounds and exact quality. For $P_0=\{p_0\}$, a sufficient geometric condition is
\begin{equation}\label{eq:obtuse-domains}
 (y-p_0)\cdot(z-p_0)\le0
 \quad(y\in\Omega_i,\ z\in\Omega_j,\ i\ne j).
\end{equation}
\end{proposition}

\begin{proof}
Independence gives $h_\Omega(P)=\max_i h_{\Omega_i}(P\cap\Omega_i)$. Each local next point is the lexicographically first maximizer of its local distance function. Among the components attaining the largest radius, the lexicographically smallest next point is therefore the lexicographically first global maximizer. The overlap assumption ensures that each new point advances only one component, and independence leaves the other local distance functions unchanged.

Write $n_i=|P\cap\Omega_i|$. Since $\sum_i n_i=N+O_I(1)$, the local generation costs, including all next-point preparation and one final export of each local state, sum to $O_I(N)$ unit-cost work and $O_I(N\log(N+2))$ bit work. There are fixedly many components, so merging and maintaining the quality measures have the same bounds. The coordinate output has the corresponding upper bound, and Lemma~\ref{lem:encoding} gives the matching lower bound in \eqref{eq:direct-cost}. The largest updated local radius gives $h(P)$, and the selected radius is the global insertion radius, which updates $\delta(P)$.

Under \eqref{eq:obtuse-domains}, an outside site $z$ satisfies $\|y-z\|^2\ge\|y-p_0\|^2$, proving independence; setting $y=z$ shows that components meet only at $p_0$.
\end{proof}

For example, $\Omega_\Pi$ and $-(3/5)\Omega_\Sigma$ meet these conditions at zero, even with different refinement levels. In contrast, AAS need not preserve its phase order under rational scaling: reduction can change the ranking by coordinate length.

\begin{proposition}[Rational scaling can change the AAS phases]
\label{prop:scaled-square}
On the single-cell square $[0,1/q]^2$, with odd $q\ge5$, let $G_k$ be the full grid of spacing $s=1/(q2^k)$ and $A_k$ its square centers. If
\begin{equation}\label{eq:badlevels}
 q\mid2^{k+2}-1,\qquad k+1\ge\lfloor\log_2q\rfloor,
\end{equation}
AAS starting from $G_k$ cannot choose exactly $A_k$ as its next phase. There are infinitely many such $k$, so the standard center--edge phase order cannot hold eventually on this fixed square.
\end{proposition}

\begin{proof}
Appendix~\ref{app:scaling-proof} constructs an admissible edge midpoint that is shorter than the last remaining center. For $q=7$, $k=1$, these are $x=(1/8,3/28)$ and $f=(3/28,3/28)$, with lengths 13 and 15.
\end{proof}

This obstruction concerns the specified AAS order; it neither determines coordinate growth on the scaled square nor excludes other optimal constructions there.

\section{Numerical experiments}\label{sec:experiments}
\subsection{Implementation and measurement protocol}\label{subsec:measurement}
The experiments examine three questions: how much local updates save
over rebuilding the Voronoi diagram, how the selection rule affects
coordinate lengths and runtime, and how much direct generation saves
when it reproduces the same ordered points.
We use PZ-Rebuild as a global-reconstruction baseline: it follows the
farthest-point rule of Pronzato and Zhigljavsky~\cite{PZ}, with
lexicographic ties, and rebuilds the restricted Voronoi diagram on the
input cover at every prefix. LEFPS uses the same geometric routines and
produces the same ordered points, but updates only affected fragments.
AAS uses its own selection rule and generally produces different points.

The timing experiments used an AMD~9800X3D processor with 32\,GB of RAM,
running Python~3.12.14, gmpy2~2.3.1, and GMP~6.3.0 under Windows~11.
Runs used one worker and one computational thread. The generic implementations maintain Voronoi fragments and use comparison
heaps. The encoded index used in the bit-complexity analysis
(Lemma~\ref{lem:priority}) is tested separately for correctness, but is
not used in these timings. On the Q4 example, generic
spatial queries scan occupied buckets rather than using the bounded-neighborhood
grid query of Section~\ref{subsec:lefps-algorithm}; the Q4 timings therefore measure this scanning implementation.
These timings do not directly test the output-sensitive bit bound.

We distinguish point output (S0), point-and-quality output (S1), and
state export (S2). S0 writes the ordered points as reduced rational
coordinates. S1 additionally writes the exact quality measures
$(h(P_n)^2,\delta(P_n)^2)$ at every prefix with $n\ge2$.
S2 saves the state needed to resume generation: fragment geometry for
the generic methods, and continuation and quality data for direct
stream execution (DSE). Every method maintains
its state through $P_N$, including PZ-Rebuild's final reconstruction.
Service times start after common imports and include parsing,
initialization, numerator--denominator text serialization, buffered
output, and flushing. The theoretical bounds use binary output.
Baseline runs include file closure; extended runs stop after the final
flush. Full-cold times additionally include process startup, request
dispatch, and a terminal state export (S2).

The nine baseline inputs use three serial repetitions. PZ-Rebuild stops
at the common prefix $N$ in Table~\ref{tab:performance}; LEFPS and AAS
continue to $2N$. The extended B2 comparison (S0, $N=1024$) and the four
generic/DSE comparisons in Table~\ref{tab:dse} use five paired repetitions
with rotating execution order. The extended stopping points were selected separately under preset
time, coordinate-height, and memory limits.
Timings are reported as medians and observed ranges; paired ratios are
$\operatorname{median}_b(T_{A,b}/T_{B,b})$. Profiling and validation are
performed separately from timing.

\subsection{Test domains and validation}\label{subsec:experimental-inputs}
B$d$ denotes the unit cube $[0,1]^d$, and S$d$ the standard simplex
$\{x\ge0:\sum_jx_j\le1\}$.
R2 is the single-cell rectangle $[0,2]\times[0,1]$; X2 is the five-square
cross~\eqref{eq:cross}, seeded at $(1,1)$. All other seeds are the origin.
H2 is the convex hexagon with vertices $(0,0),(2,0),(3,1),(2,3),(1,4),(-1,2)$.
F2 is the square frame $[0,3]^2\setminus(1,2)^2$, covered by
$[0,3]\times[0,1]$, $[0,3]\times[2,3]$, $[0,1]\times[1,2]$, and
$[2,3]\times[1,2]$. D2 comprises the triangles
$\conv\{(0,0),(1,0),(0,1)\}$ and $\conv\{(3,0),(4,0),(3,1)\}$.
I1 is the interval union $[0,1]\cup[5/3,8/3]$, and Q4 is the ordered simplex
$0\le x_4\le x_3\le x_2\le x_1\le1$. Other than the stated unions,
inputs use one convex cell, without additional triangulation.
Figure~\ref{fig:exp3} illustrates three planar geometries.
An equivalent-cover test combines $[0,2]\times[0,1]$ with either
$[0,1]\times[1,2]$ or $[0,1]\times[1/2,2]$; both covers give the same
ordered LEFPS output.

Independent exact-rational validation rebuilds the original-cover fragments.
For AAS, it tests the full candidate pool in $\kappa$-order by explicitly
intersecting edges with bisector hyperplanes, without using the local
candidate filter or the quartic predicate of Section~\ref{sec:admissible}.
Geometry and decisions are checked at every prefix through 16 points
on planar inputs, eight on B3/S3, and four on Q4, with later checks
including Table~\ref{tab:quality}. I1 is checked through 1024.
Encoded-index operations, including exclusions and restoration, are
compared separately with an exact sorted-list reference.

Timed implementations reproduce the next three points and quality pairs
after restarts at selected phase boundaries. Separate exact-rational
reference implementations restore every prefix $n=1,\ldots,128$ on R2, X2, Q4, and I1;
all reproduce the next 16 ordered points and quality pairs, covering
both phase interiors and transitions.

\begin{figure}[!htbp]
\centering
\includegraphics[width=\textwidth]{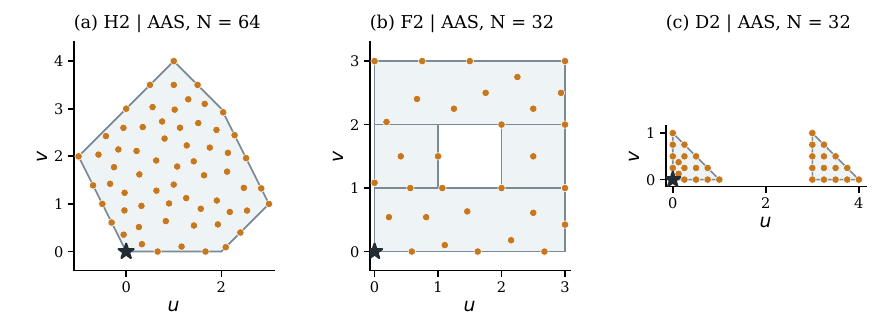}
\caption{AAS designs on (a) the convex hexagon (H2), (b) the square frame
$[0,3]^2\setminus(1,2)^2$ (F2), and (c) the disconnected triangle pair
(D2). Stars mark seeds and circles mark sites; exact inputs are defined
in Section~\ref{subsec:experimental-inputs}.}
\label{fig:exp3}
\end{figure}

\subsection{Local maintenance and policy effects}
Table~\ref{tab:performance} compares all baseline inputs; selected curves
and fragment counts appear in Figure~\ref{fig:exp4}. PZ-Rebuild and LEFPS
have identical ordered output, with paired median PZ-Rebuild/LEFPS time
ratios ranging from 2.43 to 40.76 across the inputs. On R2 at $N=128$,
PZ-Rebuild constructs a total of $8256=\sum_{n=1}^{128}n$ owner--cell fragments,
whereas LEFPS and AAS perform 707 and 582 fragment changes, respectively.

AAS is faster than LEFPS on R2 and X2 but slower on the other seven
inputs. In the extended B2 comparison at $N=1024$, LEFPS/AAS take
0.371/0.682 seconds, with paired AAS/LEFPS ratio 1.84 [1.83, 1.85],
despite output lengths of 17510/17281 bits. The slightly shorter output therefore does not translate into a runtime
advantage: candidate selection and admissibility testing also contribute
to the cost.

\begin{table}[!htbp]\centering\small\setlength{\tabcolsep}{4pt}
\caption{Baseline point-output (S0) timings at common prefixes, with three
repetitions. Times are medians [minimum, maximum] in seconds; time ratios
are medians of paired run ratios. PZ-Rebuild/LEFPS compares the same
stream, whereas AAS/LEFPS compares different streams. $L_A/L_L$ is the
output-length ratio. PZ-Rebuild reconstructs globally at each prefix;
Q4 uses the occupied-bucket scan of Section~\ref{subsec:measurement}.}
\label{tab:performance}
\begin{tabular}{@{}lrccrrr@{}}\toprule
Input & $N$ & PZ-Rebuild & LEFPS & PZ/LEFPS & AAS/LEFPS & $L_A/L_L$\\\midrule
B2 & 128 & \shortstack{1.86\\[-1pt]{\scriptsize [1.84, 1.87]}} & \shortstack{0.0454\\[-1pt]{\scriptsize [0.0453, 0.0459]}} & 40.76 & 1.79 & 0.992 \\[4pt]
B3 & 32 & \shortstack{0.591\\[-1pt]{\scriptsize [0.588, 0.601]}} & \shortstack{0.0673\\[-1pt]{\scriptsize [0.0655, 0.0697]}} & 8.78 & 1.83 & 0.738 \\[4pt]
D2 & 32 & \shortstack{0.0792\\[-1pt]{\scriptsize [0.0788, 0.0802]}} & \shortstack{0.00997\\[-1pt]{\scriptsize [0.00991, 0.011]}} & 7.90 & 1.45 & 1.000 \\[4pt]
F2 & 32 & \shortstack{0.181\\[-1pt]{\scriptsize [0.181, 0.184]}} & \shortstack{0.0183\\[-1pt]{\scriptsize [0.0182, 0.0183]}} & 9.90 & 1.73 & 0.455 \\[4pt]
H2 & 64 & \shortstack{1.53\\[-1pt]{\scriptsize [1.53, 1.53]}} & \shortstack{0.0808\\[-1pt]{\scriptsize [0.0806, 0.0817]}} & 18.95 & 1.26 & 0.810 \\[4pt]
Q4 & 8 & \shortstack{0.0229\\[-1pt]{\scriptsize [0.0229, 0.0237]}} & \shortstack{0.00964\\[-1pt]{\scriptsize [0.00945, 0.00966]}} & 2.43 & 1.62 & 1.045 \\[4pt]
R2 & 128 & \shortstack{5.61\\[-1pt]{\scriptsize [5.61, 5.66]}} & \shortstack{0.146\\[-1pt]{\scriptsize [0.145, 0.147]}} & 38.55 & 0.49 & 0.032 \\[4pt]
S3 & 16 & \shortstack{0.0518\\[-1pt]{\scriptsize [0.0517, 0.0561]}} & \shortstack{0.0125\\[-1pt]{\scriptsize [0.0125, 0.0127]}} & 4.13 & 1.64 & 0.953 \\[4pt]
X2 & 64 & \shortstack{1.43\\[-1pt]{\scriptsize [1.42, 1.43]}} & \shortstack{0.0628\\[-1pt]{\scriptsize [0.0627, 0.0636]}} & 22.72 & 0.62 & 0.035 \\[4pt]
\bottomrule\end{tabular}\end{table}

\begin{figure}[!htbp]
\centering
\includegraphics[width=\textwidth]{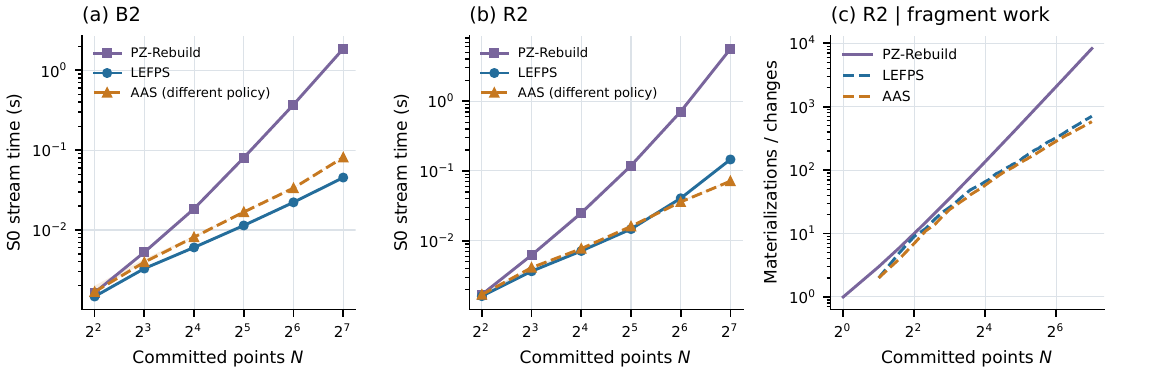}
\caption{Point-output (S0) timings on (a) the unit square (B2) and
(b) the rectangle $[0,2]\times[0,1]$ (R2), with medians and observed
ranges from three runs; (c) rectangle fragment work. PZ-Rebuild
reconstructs globally but shares LEFPS's ordered output; dashed AAS
curves follow a different policy.}
\label{fig:exp4}
\end{figure}

\subsection{Coordinate growth and arithmetic cost}
We measure total output length $L_N$ as in~\eqref{eq:output-size}
and maximum coordinate height $B_N=\max_{i\le N}\height(x_i)$.
On R2 at $N=256$, LEFPS gives $L_N=306785$ and $B_N=3937$.
For AAS, the corresponding values are $L_N=3480$ and $B_N=5$.
Figure~\ref{fig:exp5}(a)--(b) contrasts this observed growth with the
B2 sequences, whose coordinates remain short over the tested prefixes.

Separate profiles count candidate trials, edge tests, and quartic tests,
and record operand lengths in local computations and comparisons
involving nonlocal data. The counters overlap, and operand lengths are
recorded only for the instrumented operations. To isolate the effect of
operand length, we repeat a fixed batch of 32 LEFPS bisector computations
from R2 under seven rational scalings. These timings exclude input
preparation and output validation. Across seven repeated blocks, the
median time per call increases from 1.108 microseconds at an input
length of 15 bits to 4.313 microseconds at 2062 bits. This measures the
cost of the bisector routine, not its share of total sampling time.

\begin{figure}[!htbp]
\centering
\includegraphics[width=\textwidth]{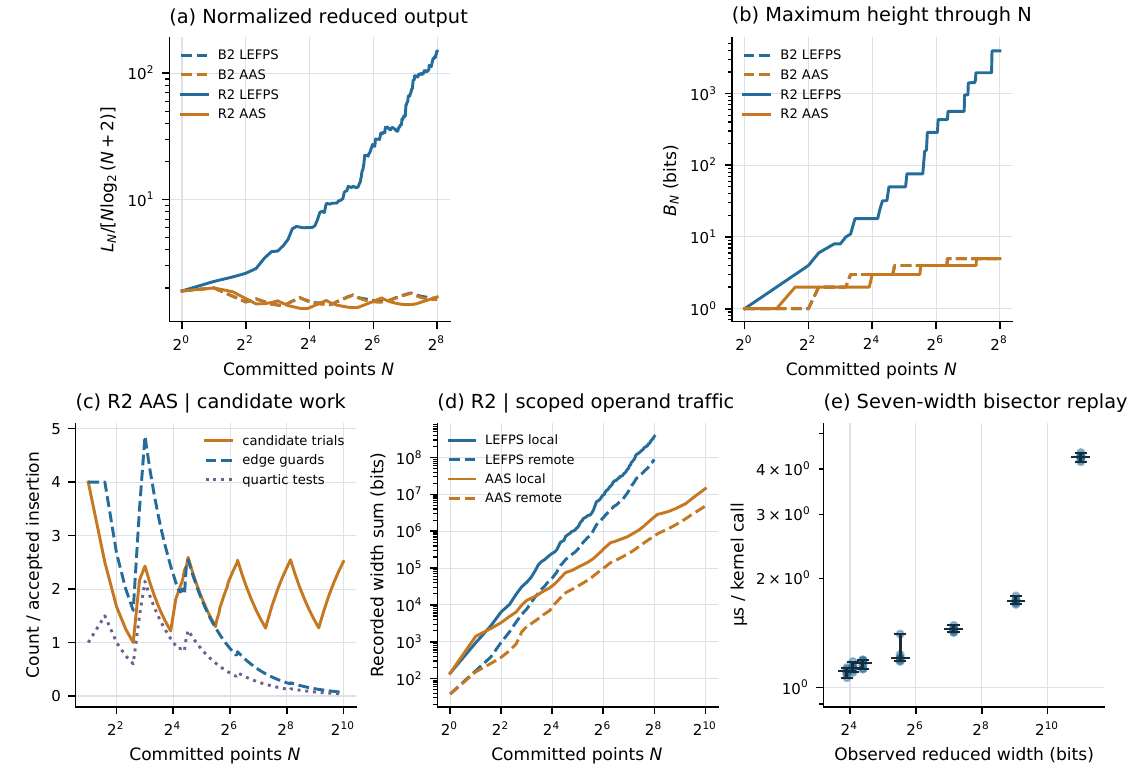}
\caption{Arithmetic diagnostics on the unit square (B2) and rectangle
$[0,2]\times[0,1]$ (R2): (a) normalized output length; (b) maximum
coordinate height; (c) candidate counts; (d) recorded operand widths;
(e) bisector-width replay. Panels (c)--(e) use R2. Profiles are separate
from timing runs; candidate counters overlap. Replay marks show seven
block observations, medians, and ranges.}
\label{fig:exp5}
\end{figure}

Table~\ref{tab:quality} reports mesh ratios and coordinate statistics at
independently checked prefixes. AAS has a smaller mesh ratio on F2 but a larger one on R2, X2, and H2;
its output is shorter on most listed inputs but longer on Q4. This is
consistent with its definition: AAS minimizes the length of the next
point within its current admissible pool, not the cumulative output
length over all possible sequences. The observed coordinate growth also
does not by itself establish an asymptotic growth rate.

\begin{table}[!htbp]\centering\small\setlength{\tabcolsep}{4pt}
\caption{Geometry and arithmetic at independently checked states. Entries are in
LEFPS/AAS order. Mesh ratios $\rho(P_n)=h_n/q(P_n)$ are rounded; $L_N$ and $B_N$ are exact.
These validation prefixes can precede the timing endpoints.}
\label{tab:quality}
\begin{tabular}{@{}lrcrr@{}}\toprule
Input & $N$ & $\rho$ & $L_N$ & $B_N$\\\midrule
R2 & 32 & 1.941 / 2.000 & 1623 / 234 & 50 / 3 \\
X2 & 32 & 1.943 / 2.000 & 2897 / 222 & 112 / 3 \\
B2 & 32 & 2.000 / 2.000 & 263 / 262 & 4 / 4 \\
H2 & 32 & 1.903 / 1.999 & 4944 / 2669 & 170 / 127 \\
F2 & 32 & 2.000 / 1.985 & 1502 / 683 & 56 / 33 \\
D2 & 32 & 2.000 / 2.000 & 255 / 255 & 4 / 4 \\
B3 & 16 & 2.000 / 2.000 & 183 / 131 & 9 / 5 \\
S3 & 16 & 2.000 / 2.000 & 129 / 123 & 3 / 3 \\
Q4 & 8 & 2.000 / 2.000 & 67 / 70 & 2 / 3 \\
\bottomrule\end{tabular}\end{table}

\subsection{Direct execution and preparation cost}
Each generic/DSE comparison uses the same selection rule and starts from
the same singleton: AAS on R2 and X2, and LEFPS on I1 and Q4. The two
implementations agree in ordered coordinates and all recorded quality
pairs. On R2 and X2, the refinement phases start after $P_6$ and $P_{12}$,
respectively; Q4 and I1 use the initial constructions described in
Section~\ref{sec:direct}. All initialization and preparation costs are
included.

\begin{figure}[!htbp]
\centering
\includegraphics[width=\textwidth]{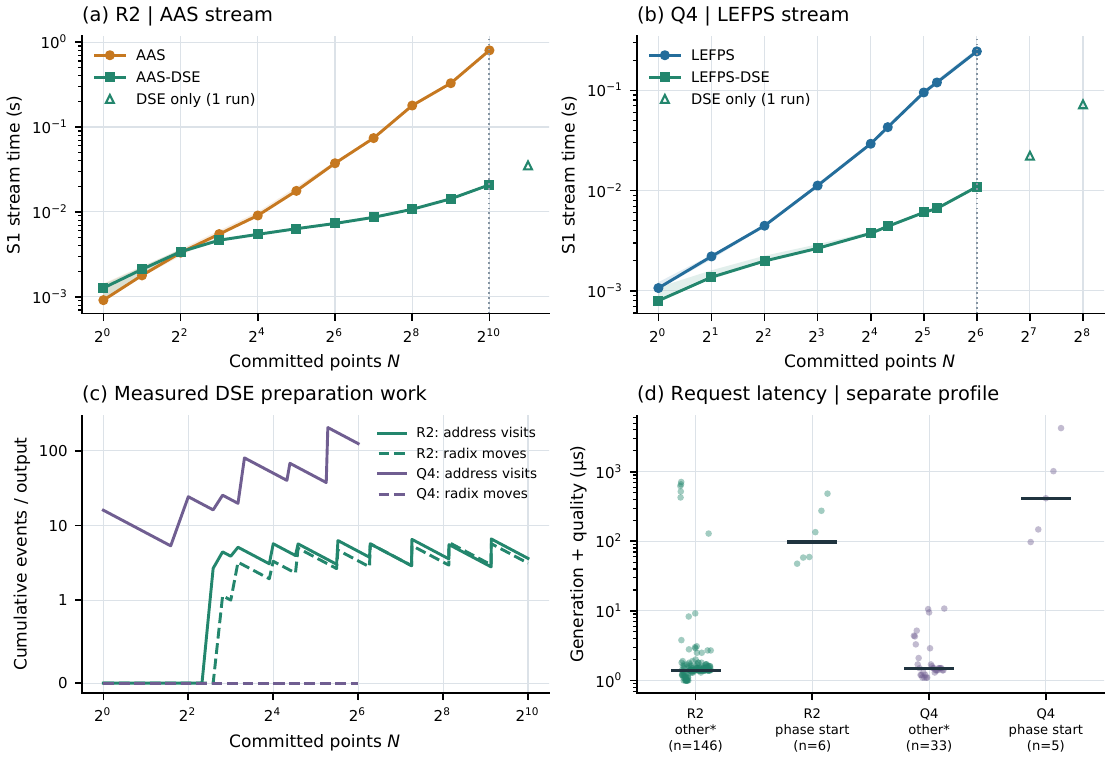}
\caption{Direct execution on the rectangle $[0,2]\times[0,1]$ (R2) and
four-dimensional ordered simplex (Q4): (a)--(b) point-and-quality (S1)
timings with five paired runs and single-run extensions (open triangles);
(c) preparation counts per output; (d) generation-and-quality latency.
Q4 generic queries scan occupied buckets.}
\label{fig:exp6}
\end{figure}

Across the four tested endpoints, the paired median generic/DSE time
ratios in Table~\ref{tab:dse} range from 7.07 to 38.23 for point-and-quality
output (S1). Including startup and native terminal-state export gives
ratios of 2.00--7.06. The single-run extensions in
Figure~\ref{fig:exp6}(a)--(b) are shown separately from the five-run comparisons.

Figure~\ref{fig:exp6}(c)--(d) shows cumulative preparation work per
output point and the time taken by individual requests. Panel~(c) uses
a linear scale near zero and a logarithmic scale above one. In
panel~(d), the group of requests that do not start a refinement phase
also includes initialization. Preparation is charged when performed,
even if some prepared points lie beyond the requested prefix. Q4
requires no radix sorting because the accepted grid indices are already
enumerated in lexicographic order. Requests that start a phase take
longer: the linear total-work bound does not give a constant worst-case
cost for each request. These latency measurements exclude serialization,
which is included in the S1 timings.

\begin{table}[!htbp]\centering\small\setlength{\tabcolsep}{4pt}
\caption{Same-policy DSE comparisons from the singleton, with five paired
repetitions. S1 times (seconds) and paired generic/DSE ratios are medians
[minimum, maximum]; full-cold ratios report medians only. Full cold
includes startup and native state export, whose contents differ between
methods. Q4 generic spatial queries scan occupied buckets.}
\label{tab:dse}
\begin{tabular}{@{}lrcccc@{}}\toprule
Input & $N$ & Generic S1 & DSE S1 & S1 ratio & Full-cold ratio\\\midrule
R2 & 1024 & \shortstack{0.799\\[-1pt]{\scriptsize [0.793, 0.834]}} & \shortstack{0.0209\\[-1pt]{\scriptsize [0.0201, 0.0211]}} & \shortstack{38.23\\[-1pt]{\scriptsize [37.69, 41.40]}} & 7.06 \\[4pt]
X2 & 512 & \shortstack{0.367\\[-1pt]{\scriptsize [0.366, 0.381]}} & \shortstack{0.0238\\[-1pt]{\scriptsize [0.0236, 0.0239]}} & \shortstack{15.54\\[-1pt]{\scriptsize [15.30, 16.04]}} & 3.68 \\[4pt]
I1 & 1024 & \shortstack{0.144\\[-1pt]{\scriptsize [0.142, 0.144]}} & \shortstack{0.0204\\[-1pt]{\scriptsize [0.0201, 0.0209]}} & \shortstack{7.07\\[-1pt]{\scriptsize [6.89, 7.13]}} & 2.00 \\[4pt]
Q4 & 64 & \shortstack{0.245\\[-1pt]{\scriptsize [0.242, 0.26]}} & \shortstack{0.0109\\[-1pt]{\scriptsize [0.0107, 0.011]}} & \shortstack{22.69\\[-1pt]{\scriptsize [21.99, 23.82]}} & 3.03 \\[4pt]
\bottomrule\end{tabular}\end{table}

\section{Discussion}\label{sec:discussion}
The two algorithms construct extensible rational designs with mesh ratio
at most two while making point-selection decisions exactly. LEFPS
preserves the prescribed farthest-point rule; AAS instead chooses a
shortest-coordinate candidate among those that preserve the mesh-ratio
bound. Their bit complexity is nearly linear in the total output length
$L_N$. Thus, the remaining arithmetic question is how large that output
can become. The interval result gives $L_N=\Theta_I(N\log(N+2))$,
but comparable bounds remain open for general higher-dimensional inputs.

\begin{problem}[Coordinate growth under LEFPS and AAS]
\label{prob:trajectory-arithmetic}
For $\pi\in\{\LEFPS,\AAS\}$, which verifiable conditions on a fixed
rational input $I$ ensure
\[
 L_N^\pi\le C_I N\log^{c_I}(N+2)
 \quad\text{for all }N,
\]
or the sharper bound $L_N^\pi=\Theta_I(N\log(N+2))$?
Are there fixed inputs in dimension at least two for which either
bound fails?
\end{problem}

Here $I$ includes the polytopal cover and rational initial point, and
coordinate lengths are measured after reduction. Theorem~\ref{thm:innovation}
gives sufficient conditions in terms of the construction depth and the
introduction and repeated use of denominator factors. Geometric quality
and dyadic coordinates alone do not control these quantities, as
Proposition~\ref{prop:long-dyadic} shows. The numerical differences between
LEFPS and AAS motivate this question but do not settle their asymptotic
growth.

A separate question concerns the best attainable complexity when the
selection rule may be chosen freely. The explicit sequences in
Section~\ref{sec:direct} admit direct generation with linear unit-cost
work and optimal bit complexity, including exact quality measures at
every prefix. How far can this be extended?

\begin{problem}[Optimal complexity on rational polytopal inputs]
\label{prob:optimal-complexity}
Which fixed rational polytopal inputs $I$ admit an exact extensible
construction that starts from the prescribed rational initial point,
satisfies $\rho(P_n)\le2$ for every $n\ge2$, and uses
\[
 O_I(N)\ \text{unit-cost work}
 \qquad\text{and}\qquad
 \Theta_I(N\log(N+2))\ \text{bit operations}
\]
at every stopping point $N$? The generator must operate without knowing
$N$ in advance. Costs include initialization, all preparation, explicit
reduced point output, and a terminal export of the state needed to
resume generation, under the conventions of Section~\ref{sec:background}.
The construction need not reproduce LEFPS or AAS.
\end{problem}

The direct constructions settle this question for the inputs treated in
Section~\ref{sec:direct}. Rational scaling can disrupt the AAS refinement
pattern (Proposition~\ref{prop:scaled-square}), but this does not exclude
an optimal construction on the scaled domain. More broadly, it remains
to identify the geometric and arithmetic properties that permit these
bounds, and whether their attainability can be decided from the rational
input description.

\section*{Declaration of competing interests}
The authors declare that they have no competing interests relevant to this work.

\section*{Funding}
The work of the second author was supported by JSPS KAKENHI Grant Number JP26K00620.

\section*{Declaration of generative AI use}
The authors used ChatGPT through ChatGPT for academic researchers (OpenAI) to assist with language editing and condensation of the manuscript. The authors reviewed and revised all AI-assisted edits, verified that the mathematical content and intended meaning were preserved, and take full responsibility for the final manuscript.

\appendix
\section{Details of the planar AAS constructions}\label{app:direct-proofs}

\subsection{Comparison with the remaining farthest points}
\label{app:planar-comparison}

\begin{proof}[Proof of the length comparison in Lemma~\ref{lem:rect-barrier}]
Let $f=\min_\kappa\Far(P)$ and let $x$ be an admissible nonfarthest
candidate. We prove $\lambda(x)>\lambda(f)$ using the original cover.

\emph{Localization.}
Put $s=2^{-k}$. Lemma~\ref{lem:rect-geometry} identifies $\Far(P)$
and $H=h(P)$ in each stated phase; grid distances certify the designs.

By Lemma~\ref{lem:focal},
\[
 \|x-f\|\le r=d_P(x)<H.
\]
The first inequality is strict whenever the domain contains a two-sided
segment through $f$. Otherwise, for one sign of a sufficiently small
$\varepsilon$, a displacement along that segment would satisfy
\[
 \|f+\varepsilon v-x\|^2
 =r^2+2\varepsilon(f-x)\cdot v+\varepsilon^2\|v\|^2>r^2,
\]
while continuity gives $d_P(f+\varepsilon v)>r$, contradicting
admissibility. Such a segment exists at centers, edge midpoints, and
missing grid vertices except convex rectangle corners; on a straight
side it is tangential.

For a common local description, put $u=x-f$ and use the following
similarity coordinates:
\[
\begin{array}{c|c}
 \text{type of }f&\xi=T(u)\\\hline
 \text{missing rectangle grid vertex}&(u_1/s,u_2/s)\\
 \text{missing square center}&((u_1+u_2)/s,(u_1-u_2)/s)\\
 \text{missing grid-edge midpoint}&(2u_1/s,2u_2/s).
\end{array}
\]
All selected sites have integer transformed coordinates, and $f$ is
the missing origin. The available adjacent axial sites are selected:
this follows from the edge-covering condition for a missing grid
vertex, and from the fixed baseline $G_k$ or $G_k\cup A_k$ in the
other cases. Comparing distances with these axial sites gives
\begin{equation}\label{eq:barrier-box}
 \xi\in(-1/2,1/2)^2
\end{equation}
intersected with the local domain. At a convex rectangle corner this
is replaced by the appropriate inward closed quadrant of
$[-1/2,1/2]^2$.
In input coordinates, the strict enclosures are respectively
\[
 \|u\|_\infty<s/2,\qquad
 |u_1|+|u_2|<s/2,\qquad
 \|u\|_\infty<s/4.
\]
At a cross boundary midpoint, $\|x-f\|<s/2$, and every original corner
is at least $s/2$ away. Thus the local domain is a half-plane.
The missing axial direction points outside it; both tangential directions
are available. Centers have all four axial neighbors, and convex rectangle
corners have the two inward neighbors.

\emph{Possible owners.}
Let $M=\|\xi\|_\infty<1/2$, and take an available axial site $e$ in
a signed coordinate direction attaining $M$. Then
\[
 \|\xi-e\|^2=\|\xi\|^2+1-2M.
\]
For a nonzero integer site $p$ that is not axial, $\|p\|_1\ge2$, and
\[
\begin{aligned}
 \|\xi-p\|^2-\|\xi-e\|^2
 &\ge\|p\|_2^2-1-2M(\|p\|_1-1)\\
 &=\bigl(\|p\|_2^2-\|p\|_1\bigr)
       +(\|p\|_1-1)(1-2M)>0.
\end{aligned}
\]
Thus only selected axial neighbors can own a point in the strict
enclosure, and their Voronoi boundaries there are diagonal rays from
the origin.
For the closed corner quadrant, $M\le1/2$ gives a nonnegative bound.
In inward coordinates all sites have nonnegative integer coordinates.
Equality for a nonaxial site is possible only for $p=(1,1)$ at
$\xi=(1/2,1/2)$. Away from that point an inward axial site is
strictly closer.

\emph{Original fragment edges.}
Any original input cell containing a candidate occurrence also contains
$f$. For centers, original facets not through $f$ are at least $s/2$
away in coordinate distance; for edge midpoints they are at least $s/2$
away, beyond the enclosure $\|u\|_\infty<s/4$. The rectangle has one cell.
Facets through $f$, including internal cross interfaces, restrict directions.
Together with the axial sites, they make $f$ the only fragment vertex in
the strict enclosure. Edges there are diagonal rays or rays on original
facets through $f$. A nonfarthest candidate must therefore be an edge
midpoint. We must locate its other endpoint in the original cell.

Consider a permitted diagonal ray $(t,t)$, $t>0$, on which $(1,0)$
and $(0,1)$ tie. A nonaxial integer site $p=(a,b)$ can meet this tie
only at
\begin{equation}\label{eq:barrier-ray}
 t=\frac{a^2+b^2-1}{2(a+b-1)}.
\end{equation}
A nonpositive denominator gives no positive finite intersection.
For a positive denominator, $0<t<1$ implies
$(a-1)^2+(b-1)^2<1$, so the only such event is $(a,b)=(1,1)$ at
$t=1/2$. At $t=1$ the possible new competitors are $(2,1)$ and
$(1,2)$.
If $(1,1)$ is unselected and no original facet has ended the edge
by $t=1$, both competitors are present and selected, as follows.

For a missing rectangle vertex, continuation beyond $(1,1)$ puts both
forward side levels at least at two. The edge-covering condition selects
$(2,1)$ and $(1,2)$, adjacent to the missing vertex $(1,1)$.
For a missing center, original sides are $\xi_1\pm\xi_2=c$ with odd $c$.
Continuation beyond $t=1$ puts the forward sum level at least at three,
and the transverse strip contains $[-1,1]$. Both competitors belong to
the cell and, by their opposite parities, to the selected $G_k$.
For a missing edge midpoint, integer side levels again include both
competitors, whose opposite parities put them in $G_k\cup A_k$.
These arguments apply separately to the original squares of the cross.

It remains to check earlier intersections with original facets.
For grid vertices and edge midpoints, a positive diagonal intersection
with an original side has integer $t$, so none occurs before one.
For square centers, a difference facet does not meet this ray, and
the forward sum facet meets it at $t=1/2$ or at least $3/2$.
In the latter case a selected-site event ends the edge by $t=1$.
Hence its first endpoint is always at $t=1/2$ or $t=1$.
The edge midpoint is $(1/4,1/4)$ or $(1/2,1/2)$, and only the former
lies in the strict enclosure. Reflections give the other permitted
signs.

On an axial original-boundary ray $(t,0)$, an integer site $(a,b)$
other than the axial owner could tie before $t=1$ only if
\[
 0<\frac{a^2+b^2-1}{2(a-1)}<1.
\]
For $a>1$ this would imply $(a-1)^2+b^2<0$, which is impossible;
for $a\le1$ there is no positive intersection before one.
Original facets not through $f$ also occur at integer axial distances.
Thus a boundary-edge midpoint lies outside the strict enclosure.
In the closed corner enclosure, the additional possibilities are an
inward axial midpoint at transformed distance $1/2$ and the diagonal
point $(1/2,1/2)$.

Returning to input coordinates gives the following exhaustive list
of necessary displacements. Signs must point into the domain; some
listed points may be absent or inadmissible.
\begin{center}
\small
\renewcommand{\arraystretch}{1.12}
\begin{tabularx}{\textwidth}{@{}>{\raggedright\arraybackslash}p{.28\textwidth}>{\raggedright\arraybackslash}p{.29\textwidth}>{\raggedright\arraybackslash}X@{}}
\toprule
Type of $f$ & Displacement $x-f$ & Boundary case\\
\midrule
Noncorner rectangle grid vertex & $(\pm s/4,\pm s/4)$ &
Strict enclosure, including straight sides.\\
Convex rectangle corner & Inward $(s/4,s/4)$, $(s/2)e_i$, or $(s/2,s/2)$ &
Closed inward enclosure.\\
Square center & $\pm(s/4)e_i$ &
Original facets have odd diagonal offsets after transformation.\\
Grid-edge midpoint & $(\pm s/8,\pm s/8)$ &
A boundary midpoint sees a half-plane; internal cover facets remain present.\\
\bottomrule
\end{tabularx}
\end{center}
Inward signs are positive at lower boundaries and negative at upper
boundaries. The list includes all possible original-edge midpoints
and the additional closed-corner vertices; it uses no fine-grid cuts.

\emph{Reduced coordinate lengths.}
Evaluate lengths in the original nonnegative coordinates.
For a reduced dyadic $v=a/2^j\ge0$ and $t>j$, every feasible change
$v\pm2^{-t}$ strictly increases its numerator-plus-denominator length.
For $v>0$, its new numerator $2^{t-j}a\pm1$ is positive, odd, and at
least $a$, while the denominator gains bits. For $v=0$, only the positive
change is feasible, with the same conclusion.

Grid-vertex coordinates have denominator exponents at most $k$,
whereas their listed changes use levels $k+1$ or $k+2$.
Center coordinates have exponents at most $k+1$ and changes at level
$k+2$; edge-midpoint coordinates have exponents at most $k+1$ and
changes at level $k+3$.
At least one coordinate changes, while unchanged coordinates retain
their lengths. Hence $\lambda(x)>\lambda(f)$ in every listed case.
The remaining farthest points are actual candidate vertices, proving
the assertion.
\end{proof}

\FloatBarrier
\subsection{The finite initial sequences}\label{app:planar-entry}

For the rectangle, Example~\ref{ex:aai-choice} gives $x_2=(1,1)$.
Before the asserted third point $(2,0)$, the earlier nonsite candidates
are $(0,1)$, $(1,0)$, $(0,1/2)$, and $(1/2,0)$.
They fail admissibility at $(2,0)$, $(2,1/2)$, $(2,0)$, and $(0,1)$,
respectively: both old-design and trial-point distances exceed the
insertion radius. The resulting $P_3$ meets every coarse grid edge,
so Lemma~\ref{lem:rect-barrier} forces the remaining points of
\eqref{eq:rect-start}. Their quality data are
\[
\begin{array}{c|ccccc}
 n&2&3&4&5&6\\\hline
 r_n^2&2&2&1&1&1\\
 h_n^2&2&1&1&1&1/2
\end{array}
\qquad h_1^2=5.
\]

For the cross, $h_1^2=5$ and $\min_\kappa\Far(P_1)=(2,3)$.
Inserting $(2,2)$ gives $r_2^2=h_2^2=2$.
The remaining farthest points are $(0,2),(2,0),(1,3),(3,1)$,
mutually separated by at least $\sqrt2$. After their insertion,
each coarse square has two opposite selected corners, and precisely
the six remaining vertices are farthest, at distance one.

To verify AAS's initial choices, set $f=\min_\kappa\Far(P)$.
Table~\ref{tab:crossfail} lists all candidates in $\|z-f\|<h(P)$
preceding the asserted choice, obtained from the original square sides
and the bisectors of the already selected points in \eqref{eq:crossstart}.
Each is rejected by the displayed witness. There are no earlier
candidates in this ball at steps 7--12. Lemma~\ref{lem:focal} excludes
all other admissible alternatives. Thus the squared covering radii are
\[
 (5,2,2,2,2,1,1,1,1,1,1,1/2),
\]
and the squared insertion radii are two at steps 2--6 and one at
steps 7--12.

\begin{table}[htbp]
\centering
\caption{Candidates preceding the asserted AAS choice during the initial
construction on the cross. In each row,
$d_{P_{n-1}}(w)^2>r_z^2$ and $\|w-z\|^2>r_z^2$, proving rejection.
All earlier candidates in the ball $\|z-f\|<h(P)$ are listed.}
\label{tab:crossfail}
\small
\begin{tabular}{@{}cccccc@{}}
\toprule
$n$ & Trial $z$ & $r_z^2$ & Witness $w$ & $d_{P_{n-1}}(w)^2$ & $\|w-z\|^2$\\
\midrule
2 & $(1,2)$ & $1$ & $(2,0)$ & $2$ & $5$\\
2 & $(1,3)$ & $4$ & $(3,2)$ & $5$ & $5$\\
2 & $(2,1)$ & $1$ & $(0,2)$ & $2$ & $5$\\
2 & $(1/2,2)$ & $5/4$ & $(2,0)$ & $2$ & $25/4$\\
2 & $(1,3/2)$ & $1/4$ & $(0,1)$ & $1$ & $5/4$\\
2 & $(3/2,1)$ & $1/4$ & $(0,1)$ & $1$ & $9/4$\\
3 & $(0,1)$ & $1$ & $(1,3)$ & $2$ & $5$\\
4 & $(1,0)$ & $1$ & $(1,3)$ & $2$ & $9$\\
5 & $(1,2)$ & $1$ & $(3,1)$ & $2$ & $5$\\
6 & $(2,1)$ & $1$ & $(3,3/2)$ & $5/4$ & $5/4$\\
\bottomrule
\end{tabular}
\end{table}

\subsection{Failure of the phase order under rational scaling}
\label{app:scaling-proof}

\begin{proof}[Proof of Proposition~\ref{prop:scaled-square}]
Put $j=k+1$, $a=\lfloor\log_2q\rfloor$,
$u=(2^j-1)/(q2^j)$, and $f=(u,u)$.
The congruence gives $\gcd(2^j-1,q)=1$.
Every center coordinate has numerator length at most $j$ and denominator
length at most $j+a+1$, and $u$ attains both bounds.
Thus $f$ is last in $\kappa$-order among the centers.

Consider $P=G_k\cup(A_k\setminus\{f\})$, where only the upper-right
fine square lacks its center.
The point $x=f+(s/4,0)$ is the midpoint of the restricted Voronoi
edge from $f$ to the right-side midpoint, so it belongs to the actual
candidate set. Its insertion and covering radii satisfy
\begin{equation}\label{eq:scaled-admissibility}
 d_P(x)^2=5s^2/16,\qquad
 h(P\cup\{x\})^2=169s^2/576<5s^2/16.
\end{equation}
To verify this exactly, translate $f$ to zero and divide by $s$.
The uncentered square is $[-1/2,1/2]^2$, with its four corners and
the new point $(1/4,0)$.
The new internal Voronoi vertices are $(-7/24,0)$ and $(0,\pm7/16)$;
the additional boundary intersections are $(1/2,\pm3/16)$.
Together with corners and side midpoints, their squared nearest-site
distances lie among
\[
 0,\quad 1/16, \quad 25/256,\quad 1/4,\quad 65/256,\quad 169/576.
\]
Hence the local squared covering radius is $169/576$.
Every other fine square is covered to distance $1/2$, and its sites
do not reduce the value at $(-7/24,0)$: the nearest neighboring
center is $17/24$ away and the neighboring grid vertices are farther.
This proves \eqref{eq:scaled-admissibility} on the whole domain.

Write $x^{(1)}$ for the first coordinate of $x$.
The divisibility assumption cancels the factor $q$:
\[
 x^{(1)}=\frac{2^{j+1}-1}{q2^{j+1}}
        =\frac{b}{2^{j+1}},\qquad
 b=\frac{2^{j+1}-1}{q}\text{ is odd},\qquad
 \bl(b)=j+1-a.
\]
The last identity follows from $2^a<q<2^{a+1}$ and $j\ge a$.
The numerator-plus-denominator lengths of $u$ and $x^{(1)}$ are
$2j+a+1$ and $2j+3-a$, respectively.
The second coordinate is unchanged, so
\begin{equation}\label{eq:scaled-saving}
 \lambda(f)-\lambda(x)=2a-2\ge2.
\end{equation}
If AAS inserted all centers before any other point, the remaining
centers would always be farthest and would be selected in
$\kappa$-order. It would therefore reach $P$ before choosing $f$.
At $P$, however, the strictly shorter admissible candidate $x$
prevents that choice. Any earlier departure also prevents completion
of the asserted center phase.
Finally, $2$ has finite multiplicative order modulo odd $q$, so
arbitrarily large multiples of this order supply $k+2$ satisfying
\eqref{eq:badlevels}.
\end{proof}


\end{document}